\newif
\ifconfver
\confvertrue        %declaring conference version true
\ifconfver
    \documentclass[10pt,journal,final,twoside]{IEEEtran}
\else
    \documentclass[11pt,draftcls,onecolumn,draftclsnofoot]{IEEEtran}
    \usepackage{setspace}
\fi
\usepackage{xspace,empheq,fancybox,amsmath,amssymb,graphicx,epstopdf,epsfig,syntonly,times,amsthm} \usepackage{psfrag,color,bm,array,cite}
\usepackage{url,cite,footnote,xspace,syntonly,algorithm,algorithmic}
\usepackage{verbatim,multirow}
\usepackage[T1]{fontenc}
\usepackage{mathtools}
\usepackage{amsmath,amsfonts,amsthm,bm} % Math packages
\usepackage{graphicx}
\usepackage{mwe}
\usepackage{caption}
\usepackage{cite}
\usepackage{amsmath,amssymb,amsfonts}
\usepackage{algorithmic}
\usepackage{graphicx}
\usepackage{textcomp}
\usepackage{xcolor}
\usepackage{multicol, blindtext}
\usepackage{xspace,empheq,fancybox,amsmath,amssymb,graphicx,epstopdf,epsfig,syntonly,times,amsthm} \usepackage{psfrag,color,bm,array,cite}
\usepackage{url,cite,footnote,xspace,syntonly,algorithm,algorithmic}
\usepackage{verbatim,multirow}
\usepackage[T1]{fontenc}
\usepackage{mathtools}
\usepackage{subfig}
\usepackage{amsmath,amsfonts,amsthm,bm}
\usepackage{graphicx}
\usepackage{mwe}
\usepackage{caption}
\usepackage{stfloats}
\usepackage{amsfonts}
\usepackage{stackengine}
\usepackage{lipsum}
\usepackage{bbm}

\usepackage{xcolor}
\usepackage{array}
\makeatletter
\newcommand{\removelatexerror}{\let\@latex@error\@gobble}
\makeatother

\newcolumntype{P}[1]{>{\centering\arraybackslash}p{#1}}

\newtheorem{remark}{\bfseries Remark}
\input{mysymbol.sty}

\definecolor{orange}{rgb}{1,0.5,0}

\renewcommand{\qed}{\hfill\blacksquare}

\DeclareMathOperator*{\esssup}{ess\,sup}

\definecolor{color_yuqi}{RGB}{0, 0, 255}

\definecolor{color_yuqi2}{RGB}{234, 135, 47}

\allowdisplaybreaks

\begin{document}
%%%%%%%%%%%%%%%%%%%%%%%%%%%%%%%%%%%%%%%%%%%%%%%%%%%%%%%%%%%%%%%%%%%%%%
%                                                                    %
%               Paper Title                                          %
%                                                                    %
%%%%%%%%%%%%%%%%%%%%%%%%%%%%%%%%%%%%%%%%%%%%%%%%%%%%%%%%%%%%%%%%%%%%%%

\title{Distributionally Robust Chance-Constrained Optimal Transmission Switching for Renewable Integration}
%\title{Distributionally Robust Chance-Constrained Programming Models for Optimal Transmission Switching Under Uncertainty}
\author{
\IEEEauthorblockN{Yuqi Zhou},~\IEEEmembership{Student Member, IEEE}, \IEEEauthorblockN{Hao Zhu},~\IEEEmembership{Senior Member, IEEE}, and \IEEEauthorblockN{Grani A. Hanasusanto}

\thanks{\protect\rule{0pt}{3mm} Manuscript received February 14, 2022; revised July 14, 2022; and accepted August 12, 2022. This work has been supported by NSF Grants 1802319 and 1752125.}

\thanks{\protect\rule{0pt}{3mm} Y. Zhou and H. Zhu are with the Department of Electrical \& Computer Engineering, The University of Texas at Austin, Austin, TX 78712, USA. G. A. Hanasusanto is with the Graduate Program in Operations Research \& Industrial Engineering, The University of Texas at Austin, Austin, TX 78712, USA; Emails: {\{zhouyuqi, haozhu, grani.hanasusanto\}{@}utexas.edu}.
}
}
%This work has been supported by NSF CAREER Grant #1802319 and NSF CAREER Grant #1752125.
% The paper headers
\markboth{IEEE TRANSACTIONS ON SUSTAINABLE ENERGY (ACCEPTED)}%
{Zhou \MakeLowercase{\textit{et al.}}: Distributionally Robust Chance-Constrained Optimal Transmission Switching for Renewable Integration}
\renewcommand{\thepage}{}
\maketitle
\pagenumbering{arabic}

\begin{abstract}
Increasing integration of renewable generation poses significant challenges to ensure robustness guarantees in real-time energy system decision-making. This work aims to develop a robust optimal transmission switching (OTS) framework that can effectively relieve grid congestion and mitigate renewable curtailment. We formulate a two-stage distributionally robust chance-constrained (DRCC) problem that assures limited constraint violations for any uncertainty distribution within an ambiguity set. Here, the second-stage recourse variables are represented as linear functions of uncertainty, yielding an equivalent reformulation involving linear constraints only. We utilize  moment-based (mean-mean absolute deviation) and distance-based ($\infty$-Wasserstein distance) ambiguity sets that lead to scalable mixed-integer linear program (MILP) formulations. Numerical experiments on the IEEE 14-bus and 118-bus systems have demonstrated the performance improvements of the proposed DRCC-OTS approaches in terms of guaranteed constraint violations and reduced renewable curtailment. In particular, the computational efficiency of the moment-based MILP approach, which is scenario-free with fixed problem dimensions, has been confirmed, making it suitable for real-time grid operations.

\end{abstract}


%\newpage

%%%%%%%%%%%%%%%%%%%%%%%%%%%%%%%%%%%%%%%%%%%%%%%%%%%%%%%%%%%%%%%%%%%%%%
% %%
% %%      Section: Intro %%
% %%%%%%%%%%%%%%%%%%%%%%%%%%%%%%%%%%%%%%%%%%%%%%%%%%%%%%%%%%%%%%%%%%%%
%%%

\section{Introduction}\label{sec:intro}

{R}{ising} renewable penetration in recent years greatly challenges the efficient and reliable operations of power systems. With the increasing  uncertainty from renewables, robust decision-making, such as grid topology optimization \cite{fisher2008optimal}, is of great importance. 
%The challenge is further compounded when the decision makings involves more complicated real-time transmission grid topology control \cite{fisher2008optimal}.
%Modern power systems can operate under 
Judicious line switching along with generation dispatch can potentially reduce generation costs and renewable curtailment level, yet at possible violations of operational limits (e.g., line power flows).  Thus, it is imperative to design optimal transmission switching (OTS) algorithms that ensure guaranteed robustness under the uncertain renewables.
%During the generation dispatch, limits imposed on control variables (e.g., generator power, transformer ratios, FACTS) have to be satisfied rigidly, while systems are allowed to sustain violations of operational constraints (e.g., power flows, bus voltages) for a certain period of time \cite[Ch. 6]{gomez2018electric}.
%Therefore, to more effectively save generation costs and manage renewable energy curtailment, it is imperative to design generation dispatch scheme under generation/demand uncertainty which can fully utilize the flexibility of the network topology while also providing certain operational robust guarantees.

The OTS problem has attracted high interest in recent years in its algorithm design and practical implementations (e.g., \cite{hedman2008optimal,kocuk2017new,ergun2012transmission,wang2018capacity}).
The switching of transmission lines expands the feasible region for generation dispatch decisions and thus relieves grid congestion.
Therefore, it can potentially adapt to the varying power transfer needed for  renewable generation and reduce the level of renewable curtailment. For example, over 7000 MW of wind capacity was installed in  Texas from 2006 to 2009, but major transmission congestion was experienced~\cite{EIA2}. Transmission constraints have resulted in  excessive wind curtailment, as reported in \cite{NREL,CAISO2}. Thus,  efficiently solving OTS is of great importance to enhance the penetration level of renewable energy to the grid. %To avoid excessive curtailment, the state changed its transmission infrastructures, enabling the wind curtailment to steadily drop from 17\% in 2009 to 0.5\% in 2014 \cite{NREL}.
%reports from: https://blog.ucsusa.org/mark-specht/renewable-energy-curtailment-101/
%https://www.eia.gov/todayinenergy/detail.php?id=16831
%http://www.caiso.com/Documents/BoardApproved2020-2021TransmissionPlan.pdf
%It has been also reported by National Renewable Energy Laboratory (NREL) and California Independent System Operator (CAISO) that the major causes for the excessive curtailment are congestion and local transmission limitations \cite{CAISO2, NREL}. Therefore, performing efficient operations on the transmission network can help relieve network congestion and save unnecessary renewable curtailment.
%{\yuqi
%(Continue from here)
%}

As the OTS problem includes additional integer decision variables, extending it to a stochastic/robust optimization framework is more difficult than that for optimal power flow (OPF). Similar to OPF, scenario-based approaches have been largely used to deal with the OTS problem under uncertainty. For example, stochastic topology optimization has been considered in  \cite{shi2016wind,dehghanian2016probabilistic} based on wind uncertainty scenarios from known probabilistic models. 
In addition, the chance-constrained (CC) framework has been developed for the OTS problem in \cite{qiu2015chance}, aiming to attain guaranteed constraint violation for a given uncertainty distribution using the sample-average approximation (SAA) approach. Nonetheless, constructing an accurate distribution for the uncertainty in energy resources can be extremely challenging in practice. Moreover, these approaches typically lead to a mixed-integer program (MIP) in which the problem dimensions quickly grow with the number of samples. This scalability issue results in high computational complexity and makes these scenario-based approaches sub-par for real-time OTS decision-making.

To tackle the scalability issue with scenario-based approaches, some robust/stochastic OTS work \cite{korad2013robust,dehghan2015robust,lan2021stochastic} invokes a repeating procedure of adding cuts to a master problem using sub-problem solutions. Nonetheless, their computational efficiency can still be problematic while the optimality guarantee is unclear. 
Instead, our earlier work \cite{zhou2020transmission} has proposed a robust OTS algorithm by using the linear decision rule (LDR) technique to approximate  second-stage variables, seeking to maintain the operating limits for any uncertainty within a compact set. Even though LDR constitutes merely a linear approximation, the resultant mixed-integer linear program (MILP) has a fixed problem dimension and is efficient to solve.
Nonetheless, its robustness under all possible uncertainty scenarios makes the solution unnecessarily conservative. 
%{\yuqi This LDR based solution has successfully addressed the complexity issue incurred by existing robust/stochastic OTS work \cite{korad2013robust,dehghan2015robust,lan2021stochastic} by repeatedly adding cuts to a master problem. 
%Although these techniques provide more scalable solutions to the OTS problem than scenario-based approaches, the computational efficiency cannot always be guaranteed and the conservativeness of the formulation remains a concern. }
In addition, recent work \cite{nazemi2019mixed} has considered the distributionally robust chance-constrained (DRCC) OTS problem to account for the ambiguity of uncertainty distribution. Nonetheless, the linearized OTS model therein builds upon line outage sensitivity factors and cannot accurately include multiple, simultaneous topology changes.

Our work aims to provide computationally efficient algorithms for solving the DRCC-OTS problem by developing equivalent reformulations. Notably, we consider an \textit{equivalent} linear reformulation for the integer line status and dc power flow variables. We analytically establish that for the two-stage OTS under linear generation response, the recourse actions (line flows and phase angles) can be represented as linear functions of the uncertainty variables. This linear OTS model is crucial for tractable DRCC reformulation through  dualization. 
Compared to CC-OTS, the proposed DRCC framework seeks  dispatch and switching decisions that are robust against the worst-case uncertainty distribution from within a prescribed ambiguity set. Thus, it greatly expands the possible probabilistic models, where variations of uncertainty distributions are common in real-world settings due to the lack of data samples or high variability. 
The DRCC approaches are of particular importance for enhancing renewable integration because they can provide guaranteed robustness performance as demonstrated by our numerical tests.
% seeks dispatch and switching decisions under all possible probability distributions within the ambiguity set.
%Moreover, by considering the partial information on the underlying distributions, the DRCC generally adopts conic reformulations, which also guarantees the computational efficiency as compared to CC-OTS.
To provide tractable DRCC-OTS solutions, this work considers moment-based (mean-mean absolute deviation) and distance-based (Wasserstein distance) ambiguity sets for the renewable uncertainty, both of which are amenable to linear reformulations.

% We formulate it as a real-time sequential decision making problem (see also in \cite{zhang2016distributionally}) by utilizing two-stage robust optimization. A here-and-now decision is made at the first stage, and then recourse actions are taken at the second stage upon observing the generation/demand uncertainty.
% Chance constrained programs (CCPs) are normally difficult to solve, and it can become more complicated when the problem involves integer variables. 
%In addition, we use two approximation methods (sample average approximation and Gaussian approximation) for CC-OTS to serve as benchmark to the DRCC approach.
The contribution of our work is three-fold.
\begin{itemize}
\item We put forth a two-stage DRCC-OTS problem under renewable uncertainty that models real-time linear adjustment of generation output.
\item For the proposed two-stage OTS problem, we analytically establish an equivalent LDR-based reformulation by recognizing that recourse actions are exactly linear in the uncertainty for given first-stage decision variables.  
\item We are the first to construct scalable DRCC-OTS problems using the mean dispersion and the ${\infty}$-Wasserstein ambiguity sets, both leading to MILP reformulations through dualization-based analysis. 
\end{itemize}
{Numerical tests demonstrate the proposed DRCC-OTS solutions can effectively limit the constraint violations and reduce curtailment under renewable uncertainty, greatly improving the robustness guarantees over CC-OTS.} Furthermore, the moment-based DRCC-OTS approach is scenario-free and efficiently solvable, hence very suitable for real-time grid operations.

% The two-stage DRCC-OTS problems allows for certain violations of operational limits under uncertain generation/demand deviations while considering the flexibility of system network topology, thus leading to more efficient topology control and generation dispatch solutions. The mixed-integer linear program (MILP) reformulations provided in this work can be efficiently solvable and considered for real-time grid operations.

% The chance constrained (CC) formulation has been applied to various power system generation dispatch problems.
% The CC optimal power flow (CC-OPF) problem has been discussed broadly in many works \cite{bienstock2014chance,zhang2011chance,dall2017chance,roald2016corrective,roald2017chance,venzke2017convex,wang2011chance} in the past. Moreover, by considering the reconfiguration of network topology, the CC optimal transmission switching (CC-OTS) problem has been studied in \cite{qiu2015chance,}.

% DRCCO in optimal power flow \cite{zhang2016distributionally,guo2018data,xie2017distributionally,duan2018distributionally,lubin2015robust}

The rest of the paper is organized as follows. Section \ref{sec:ps} formulates the OTS problem based on the dc power flow model. 
Section \ref{sec:ccots} introduces the linear equivalent reformulation of the two-stage robust OTS problem. 
For comparison purposes, two benchmark CC-OTS approaches with conic reformulations are considered and a linear program is further presented for quantifying the benefits of renewable curtailment for each given approach. 
Section \ref{sec:drcc} presents the DRCC-OTS algorithms using both moment-based and distance-based ambiguity sets. 
Numerical experiments using the IEEE 14-bus and 118-bus systems are presented in Section \ref{sec:num} to demonstrate the improvements of the proposed DRCC-OTS algorithms in terms of guaranteed robustness and computational efficiency.
The paper is wrapped up in Section \ref{sec:con}.

\textit{Notation:} 
Bold symbols stand for matrices/vector and unbolded symbols stand for scalars; $(\cdot)^{\mathsf T}$ stands for transposition; $|\, \cdot \,|$ denotes the absolute value; $\left\| \, \cdot \, \right\|$ denotes the vector norm; $\circ$ denotes the Hadamard product; $\mathbf{e}$ denotes the vector of all ones; $\bbe_i$ denotes the standard basis vector with all entries being 0 except for the $i$-th entry equals to 1; $\mathbbm{1}$ denotes the indicator function; $\mathcal{M}_{+}$ denotes the set of nonnegative measures.

\section{System Modeling}\label{sec:ps}

We first  present the dc power flow based optimal transmission switching (OTS) formulation \cite{fisher2008optimal}. Consider a transmission system with $N$ buses collected in the set $\cal N :=$ $\{1,\ldots,N\}$ and $L$ lines in $\cal L :=$ $\{(i,j)\} \subset \cal N \times \cal N$.
Let $\theta_i$ denote the voltage angle per bus $i$ and the vector $\bm{\theta} \in \mathbb{R}^{N}$ collect all $\theta_i$'s. Similarly, let $\bm{g},~\bm{d} \in \mathbb{R}^{N}$ denote the vectors of nodal generation and load, respectively. The line flow $\bm{f}:=[\{f_{ij}\}] \in \mathbb{R}^{L}$ becomes %are collected in $\bm{f} \in \mathbb{R}^{L}$, given by:
\begin{align}
\bm{f} = \mathbf{K}\bm{\theta} \label{eq:PF1}
\end{align}
where the matrix $\mathbf{K} \in \mathbb R^{L\times N}$ is formed by the topology and line parameters. Specifically, its row for line $(i,j)$ equals to $b_{ij} (\mathbf{e}_i-\mathbf{e}_j)^\mathsf T$, with $b_{ij}$ being the inverse of line reactance and $\mathbf{e}_i$ the $i$-th standard basis vector. Furthermore, the nodal power balance leads to the total injection $\bm{p} := \bm{g} - \bm{d}$ as:
\begin{align}
\bm{p} = \mathbf{A}\bm{f} \label{eq:PF2}
\end{align}
where $\mathbf{A} \in \mathbb{Z}^{N \times L}$ corresponds to the graph incidence matrix for $(\cal N, \cal L)$, with the column for line $(i,j)$ set to $(\mathbf{e}_i-\mathbf{e}_j)$.

The OTS problem aims to determine the connectivity of transmission lines so as to minimize the total generation cost for a given load $\bm{d}$. For simplicity, we consider a linear generation cost (as in \cite{fisher2008optimal,hedman2008optimal,heidarifar2015network}) and use $\bm{c} \in \mathbb{R}^{N}$ to denote the vector of (known) linear cost coefficients. In addition to the dispatch $\bm g$, the OTS's decision variables include a binary vector $\bm{z} \in \mathbb{R}^{L}$ to indicate the transmission line status (1: closed, 0: open). The OTS problem is formulated as a mixed-integer linear program (MILP), given by
\begin{subequations} \label{eq:TS}
\begin{align}
\min \quad & {\bm{c}^{\mathsf T}\bm{g}}\\
\textrm{s.t.} \quad &  \bm{g} \in \mathbb{R}^{N}, \bm{\theta} \in \mathbb{R}^{N}, \bm{f} \in \mathbb{R}^{L}, \bm{z} \in \mathbb{Z}^{L}\label{eq:OTS_b}\\ 
  & \underline{\bm{g}} \leq \bm{g} \leq \overline{\bm{g}}    \label{eq:OTS_c}\\
  & \underline{\boldsymbol{\theta}} \leq \boldsymbol{\theta} \leq \overline{\boldsymbol{\theta}} \label{eq:OTS_d}\\
  & \underline{\bm{f}} \circ \bm{z} \leq \bm{f} \leq \overline{\bm{f}} \circ \bm{z} \label{eq:OTS_e}\\
  & \mathbf{A}\bm{f} = \bm{g} - \bm{d}\label{eq:OTS_f}\\
  & \mathbf{K}\boldsymbol{\theta}  - \bm{f} + {\mathbf{M}} \circ (\mathbf{e} - \bm{z}) \geq {\bm{0}}\label{eq:OTS_g}\\
  & \mathbf{K}\boldsymbol{\theta}  - \bm{f} - {\mathbf{M}} \circ (\mathbf{e} - \bm{z}) \leq {\bm{0}}\label{eq:OTS_h}\\
  & \mathbf{e}^{\mathsf T}\bm{z} \geq L - L_o \label{eq:OTS_i}
\end{align}
\end{subequations}
where constraints on generation, angle and line flow in \eqref{eq:OTS_c}-\eqref{eq:OTS_e} enforce the system operating limits. We use $\circ$ to denote the component-wise product (or Hadamard product), which is used in \eqref{eq:OTS_e} to enforce the limits on closed lines only according to $\bm{z}$. %vectors $\underline{\bm{f}}$ and $\overline{\bm{f}}$ collect the given lower/upper limits of line flow, respectively. 
For any open line ($z_{ij}=0$), its flow $f_{ij}$ becomes zero under \eqref{eq:OTS_e}. Additionally,  the constraint \eqref{eq:OTS_f} enforces network power balance as in \eqref{eq:PF2}. As for constraints \eqref{eq:OTS_g} and \eqref{eq:OTS_h}, they jointly represent the line flow model in \eqref{eq:PF1} where the vector ${\mathbf{M}}$ has each entry $\mathrm{M}_{ij}$ for each line $(i,j)$ to be a sufficiently large constant.
For any closed line ($z_{ij}=1$), the two inequalities exactly lead to the equality constraint as in \eqref{eq:PF1}. Otherwise, under $z_{ij}=0$ and thus $f_{ij}=0$ [cf. \eqref{eq:OTS_e}], the two constraints respectively become $b_{ij}(\theta_{i} - \theta_{j}) + \mathrm{M}_{ij} \geq 0$ and $b_{ij}(\theta_{i} - \theta_{j}) - \mathrm{M}_{ij} \leq 0$. Accordingly, both conditions trivially hold under a large enough  $\mathrm{M}_{ij}$ and does not affect the OTS problem. This technique is known as the \textit{Big-M} method~\cite{griva2009linear}, which is powerful for handling constraints with binary variables.
For each line $(i,j)$, we can set $\mathrm{M}_{ij} \coloneqq b_{ij} \Delta \theta_{ij}^{\max}$ with a maximum limit $\Delta \theta_{ij}^{\max}$ according to angle stability [cf. \eqref{eq:OTS_d}]. Lastly, for system stability concerns, we impose the constraint \eqref{eq:OTS_i} to restrict the total number of lines that can be switched off not to exceed the given limit $L_o$. In fact, this restriction can also reduce the computational complexity of solving the resultant MILP. In addition, earlier studies (see e.g., \cite{fisher2008optimal,qiu2015chance,heidarifar2015network}) have shown that the incremental reduction of total generation cost diminishes rapidly when $L_{o}$ reaches a certain level. Practical choices of $L_o$ are relatively small (e.g., $L_o\leq 4$) for large systems.

\begin{remark}[{power flow modeling}] This paper adopts the dc power flow model for formulating the OTS problem. Albeit simple, it does not include voltage limits or other ac flow considerations. To address this, it is possible to extend to the ac power flow by using the relaxation-based formulation in~\cite{kocuk2017new}. In addition, one can perform the post-selection ac flow analysis and verify the ac feasibility of the resultant solution to \eqref{eq:TS}, as introduced in \cite{goldis2015ac}.
\end{remark}

While the dc-OTS solutions may not always be ac feasible as pointed out by earlier papers (e.g., \cite{stott2009dc,shchetinin2018construction,crozier2022feasible}), there exist some corrective measures to attain ac-feasiblility; see e.g., \cite{barrows2014correcting} and references therein. For example, one can try to remove one single line from the dc-optimal solution of switched lines in order to maintain the satisfaction of constraints. The selection of the line removal could depend on the reactance/resistance criteria as proposed in \cite{barrows2014correcting}. This screening process could be repeated until an ac-feasible solution has been obtained. While these solutions do not exactly guarantee ac-feasibility, they turn out to be very effective in practice \cite{barrows2014correcting,bai2016two,coffrin2014primal}.

\section{OTS Under Uncertainty}
\label{sec:ccots}
This section formally presents the OTS problem under uncertainty as well as its chance-constrained solutions.
We first discuss the model of uncertainty due to e.g., renewable generation or flexible demand. Let $\bbxi\in \mathbb R^K$ stand for the  uncertainty vector of the full system  with its samples  denoted by $\{{\boldsymbol{\xi}}^{j}\}_{j=1}^S$. 
% {\hao (isn't the f-d sth we assumed? if so, how do we directly `define' it? if f-d is a trivial simplification, we should say `The convex hull of these points forms the support set for $\bbxi$, which is typically assumed to be full-dimensional without loss of generalizability (wlog) \cite{habets2004control}, as detailed here.')}
% DR. Hanasusanto's revision below
{We assume that $\bm\xi$ is bounded with a certain support set. In a data-driven setting, the set can be estimated  with high confidence from the samples under mild assumptions on the distribution (e.g., sub-Gaussian). For example, it can be the polytope formed by  the convex hull of the samples \cite{gazijahani2017robust}. The following general condition is assumed.}
%The convex hull of these points forms the support set for $\bbxi$, which is typically assumed to be full-dimensional \cite{habets2004control} without loss of generalizability (wlog).
% A full-dimensional polytope \cite{habets2004control} is defined as the convex hull of these points. 
%For ease of exposition, in this work the following assumption is made on the uncertainty under the full-dimensional model:
\begin{assumption}\label{as:1}
The support set for $\bbxi$ is compact and represented by a full-dimensional polytope $\boldsymbol\Xi \coloneqq \{\boldsymbol{\xi} \in \mathbb{R}^{K}: \mathbf{U}\boldsymbol{\xi} \leq \mathbf{t}\}$.
\end{assumption}
% {\hao (what's the implication of AS1? why is it 
%This assumption turns out to be useful for representing OTS decision variables under uncertainty. 
To incorporate the uncertainty into \eqref{eq:TS}, we resort to a two-stage robust optimization 
% (see also \cite{shao2017security,yan2018robust,ding2016two}) 
by making a \emph{here-and-now} decision while taking \emph{recourse} or \emph{wait-and-see} actions once the realizations of $\boldsymbol{\xi}$ are observed. 
Recourse functions are defined for the generation, angle, and line flow variables upon observing $\bbxi$.  
For simplicity,  we consider a linear response modeling for generator recourse actions, as motivated by frequency response and automatic generation control mechanisms \cite[Ch.~9]{wood2013power}. 
\begin{assumption}\label{as:2}
The generation recourse actions follow a linear response mechanism that adjusts each dispatchable generator by a fixed percentage of instantaneous network-wide power imbalance. As the latter is equal to  $\mathbf{e}^{\mathsf T} \boldsymbol{\xi}=\xi_1+\ldots+\xi_K$, the generation adjustment  becomes ${\bm{g}}'(\boldsymbol{\xi}) = \boldsymbol{\gamma} (\mathbf{e}^{\mathsf T} \boldsymbol{\xi})$, with vector $\boldsymbol{\gamma} \in \mathbb R^N$  collecting the linear coefficients to be determined.
\end{assumption}

This linear policy has been widely adopted by various earlier work (e.g., \cite{bienstock2014chance,roald2017chance,lorca2016multistage}), as it can quickly restore the system-wide power imbalance.
Specifically, the recourse actions are linear functions of total power mismatch, which can be quickly corrected by a proportional change from each generating unit. Such policy is very convenient to implement in practice as system-wide power mismatch is easily measured using frequency deviation. As a result, area-wide frequency responses require minimal communication overhead.

The flexibility of generation output is  limited by the committed reserves, with $\overline{\bm{r}},\underline{\bm{r}}\in \mathbb R^N$ denoting its upper/lower limits.
Moreover, changes of angles and line flows are respectively denoted by ${\boldsymbol{\theta}}'(\boldsymbol{\xi}):\mathbb R^K\rightarrow \mathbb R^N$ and ${\bm{f}}'(\boldsymbol{\xi}):\mathbb R^K\rightarrow \mathbb R^L$, both as recourse functions of $\boldsymbol{\xi}$. Inspired by the linearity of dc power flow, we will model them as linear functions, i.e., we have $\boldsymbol{\theta}'(\boldsymbol{\xi})=\bm{Y}_{\theta}{\boldsymbol\xi}$ and ${\bm{f}}'(\boldsymbol{\xi})= \bm{Y}_{f}{\boldsymbol\xi}$ with matrices ${\bm{Y}}_{{\theta}} \in \mathbb{R}^{N \times K}$ and $\bm{Y}_{f} \in \mathbb{R}^{L \times K}$ as decision variables. This approach is well known as the linear decision rule (LDR) scheme in two-stage robust optimization \cite{kuhn2011primal}, which approximates the recourse variables as affine functions of uncertainty. Interestingly, under (AS\ref{as:1})-(AS\ref{as:2}) this linearized model turns out to be \textit{exact} in representing the actual changes of angles and line flows at no modeling error, as detailed shortly. With vector $\bm{q}$ collecting the linear cost coefficients for generation adjustment, the OTS problem under uncertain $\boldsymbol{\xi}$ is cast as:
\begin{subequations}\label{eq:LDR_2}
\begin{align}
\min \quad & \bm{c}^{\mathsf T}\bm{g} + \mathbb{E}[\bm{q}^{\mathsf T}\boldsymbol{\gamma} \mathbf{e}^{\mathsf T} \boldsymbol{\xi}] \label{eq:LDR_a2}\\
\textrm{s.t.} \quad & \eqref{eq:OTS_b} - \eqref{eq:OTS_i},\; \boldsymbol{\gamma} \in [0,1]^{N},\;\mathbf{e}^{\mathsf T} \boldsymbol{\gamma} = 1 \label{eq:LDR_b2}\\
  & {\bm{Y}}_{{\theta}} \in \mathbb{R}^{N \times K},\; \bm{Y}_{f} \in \mathbb{R}^{L \times K} \label{eq:LDR_c2}\\
  & \underline{\bm{r}}  \leq \boldsymbol{\gamma} \mathbf{e}^{\mathsf T} \boldsymbol{\xi} \leq \overline{\bm{r}}\;\; \forall {\boldsymbol\xi} \in  \boldsymbol{\Xi}\label{eq:LDR_d2}\\
  & \underline{\bm{g}} \leq \bm{g} + \boldsymbol{\gamma} \mathbf{e}^{\mathsf T} \boldsymbol{\xi} \leq \overline{\bm{g}} \;\; \forall {\boldsymbol\xi} \in  \boldsymbol{\Xi}\label{eq:LDR_e2}\\
  & \underline{\boldsymbol{\theta}} \leq \boldsymbol{\theta} + \bm{Y}_{\theta}{\boldsymbol\xi} \leq \overline{\boldsymbol{\theta}} \;\; \forall {\boldsymbol\xi} \in  \boldsymbol{\Xi}\label{eq:LDR_f2}\\
  & \underline{\bm{f}} \circ \bm{z} \leq \bm{f} + \bm{Y}_{f}{\boldsymbol\xi}  \leq \overline{\bm{f}} \circ \bm{z} \;\; \forall {\boldsymbol\xi} \in  \boldsymbol{\Xi}    \label{eq:LDR_g2}\\
  & \mathbf{A}({\bm{f} + \bm{Y}_{f}{\boldsymbol\xi}} ) = \bm{g} + \boldsymbol{\gamma}\mathbf{e}^{\mathsf T}{\boldsymbol\xi} - \bm{d} - \mathbf{F}{\boldsymbol\xi} \;\; \forall {\boldsymbol\xi} \in  \boldsymbol{\Xi}\label{eq:LDR_h2}\\
  & \mathbf{K}(\boldsymbol{\theta} + \bm{Y}_{\theta}{\boldsymbol\xi} ) - \bm{f} - \bm{Y}_{f}{\boldsymbol\xi} + {\mathbf{M}}\circ(\mathbf{e} - \bm{z}) \geq \bm{0}\;\; \forall {\boldsymbol\xi} \in  \boldsymbol{\Xi}\label{eq:LDR_i2} \\
  & \mathbf{K}(\boldsymbol{\theta} + \bm{Y}_{\theta}{\boldsymbol\xi}) - \bm{f} - \bm{Y}_{f}{\boldsymbol\xi} - {\mathbf{M}}\circ(\mathbf{e} - \bm{z}) \leq \bm{0}\;\; \forall {\boldsymbol\xi} \in  \boldsymbol{\Xi}.\label{eq:LDR_j2}
\end{align}
\end{subequations}
In the following, we define  $\mathbf{x}$ to be the vector of  decision variables comprising the first-stage decisions $(\bm{g},\bm{\theta},\bm{f},\bm{z},\boldsymbol{\gamma})$. Note that the coefficient $\bm{\gamma}$ is included for reducing the total cost. The second-stage decisions ${\bm{Y}}_{{\theta}}$ and $\bm{Y}_{f}$ relate the  angle and line flow adjustments to $\boldsymbol{\xi}$. 
Moreover, the transformation matrix $\mathbf{F} \in \mathbb{R}^{N \times K}$ in \eqref{eq:LDR_h2} is a known mapping  from $\boldsymbol{\xi} \in \mathbb{R}^{K}$ to the system dimension $N$. Basically, problem \eqref{eq:LDR_2} aims to minimize the sum of total generation cost at the first-stage and the expected cost during real-time \textit{recourse} adjustment. Constraint \eqref{eq:LDR_d2} imposes the operating reserve limits while the remaining ones \eqref{eq:LDR_e2}-\eqref{eq:LDR_j2} ensure that the system operating limits  in the OTS problem \eqref{eq:TS} would still hold after recourse actions are taken [cf. \eqref{eq:OTS_c}-\eqref{eq:OTS_h}]. With the mean of uncertainty $\mathbb{E}[\boldsymbol{\xi}] = \boldsymbol \mu$ known,  the term $\mathbb{E}[\bm{q}^{\mathsf T}\boldsymbol{\gamma} \mathbf{e}^{\mathsf T} \boldsymbol{\xi}]$ in \eqref{eq:LDR_a2} simplifies to $\bm{q}^{\mathsf T}\boldsymbol{\gamma} \mathbf{e}^{\mathsf T} \boldsymbol{\mu}$, which is a linear function of the unknown $\bbgamma$.

By recognizing this linearity, we can establish the exactness of the modeling on $\boldsymbol{\theta}'(\boldsymbol{\xi})=\bm{Y}_{\theta}{\boldsymbol\xi}$ and ${\bm{f}}'(\boldsymbol{\xi})= \bm{Y}_{f}{\boldsymbol\xi}$, as follows. 
%{\yuqi 
%The lemma holds true under AS\ref{as:1} and AS\ref{as:2}, where the linear AGC scheme is recognized under the full-dimensional support set of uncertainty.}
\begin{lemma}
Under (AS\ref{as:1})-(AS\ref{as:2}),  the adjustments on angle and line flow for fixed grid topology $\bm{z}$ and coefficients $\bbgamma$ become exactly linear functions of  $\boldsymbol \xi$, with both matrices $\bm{Y}_\theta$ and $\bm{Y}_{f}$ uniquely determined by $\bm{z}$ and $\boldsymbol{\gamma}$. \label{eq:theorem}
\end{lemma} 
\begin{proof}
Under (AS\ref{as:2}), the change of full network injection due to uncertainty is $(\boldsymbol{\gamma} \mathbf{e}^{\mathsf T} - \mathbf{F}){\boldsymbol\xi}$ [cf. \eqref{eq:LDR_h2}].
Under the fixed topology of no islanding, the dc linear flow model \cite[Ch. 4]{wood2013power} states that changes of angle and line flow, namely $\bbtheta'$ and ${\bm{f}}'$, are linearly related to the change of injection, and thus to $\boldsymbol{\xi}$ as well. 
% {\hao (do we use (AS1) here for sth, like uniqueness? o.w., if say $\mathbf 1^T \bbxi=0$, then there could be multiple $Y_f$ or $Y_\theta$? if that wasn't an issue, then we don't need (AS1). basically we have to clearly state what assumptions are needed for what claim)}
Lastly, the full-dimensionality of the support set $\bbXi$ in  (AS\ref{as:1}) further guarantees the uniqueness of $\bm{Y}_\theta$ and $\bm{Y}_{f}$. %Of course, the mapping from $\bbxi$ to either $\bbtheta'$ or ${\bm{f}}'$ would depend on the topology in $\bm{z}$ and AGC coefficients in  $\boldsymbol{\gamma}$.
$\qed$
\end{proof}

Lemma \ref{eq:theorem} ensures the linear models in problem \eqref{eq:LDR_2} produce the exact recourse values for angle and line flow under given first-stage decision variables of $\bm{z}$ and $\bbgamma$. 
Hence, our LDR approach yields an exact model for the recourse variables, and problem \eqref{eq:LDR_2} constitutes an \textit{equivalent} two-stage OTS formulation. This is a much stronger result than existing  LDR solutions \cite{kuhn2011primal}, including the earlier OTS application in \cite{nazemi2019mixed}.

\subsection{Chance-Constrained (CC-) OTS}
The chance-constrained (CC) formulation is popularly employed to deal with inequality constraints under uncertainty~\cite{li2021coordinating}. 
It ensures that constraints are satisfied with probability above a prescribed threshold.
The relevant constraints from problem~\eqref{eq:LDR_2} can be collected in the following set:
\begin{align} \label{eq:set}
   \ccalI =  \Big\{ & \underline{\bm{r}}  \leq \boldsymbol{\gamma} \mathbf{e}^{\mathsf T} \boldsymbol{\xi} \leq \overline{\bm{r}}, \nonumber\\
  & \underline{\bm{g}} \leq \bm{g} + \boldsymbol{\gamma} \mathbf{e}^{\mathsf T} \boldsymbol{\xi} \leq \overline{\bm{g}},\nonumber\\
  & \underline{\boldsymbol{\theta}} \leq \boldsymbol{\theta} + \bm{Y}_{\theta}{\boldsymbol\xi} \leq \overline{\boldsymbol{\theta}},\nonumber\\
  &  \underline{\bm{f}} \circ \bm{z} \leq \bm{f} + \bm{Y}_{f}{\boldsymbol\xi}  \leq \overline{\bm{f}} \circ \bm{z}  \Big\}.
\end{align}
These constraints correspond to the limits on reserve, generation, phase angle, and line flow as in \eqref{eq:LDR_d2} - \eqref{eq:LDR_g2}, all of which are linear in $\bbxi$. Note that the network power balance in \eqref{eq:LDR_h2} and line flow relations in  \eqref{eq:LDR_i2} - \eqref{eq:LDR_j2} are not part of the set \eqref{eq:set}. This is because they are used to determine the power flow  and thus need to be satisfied strictly. Interestingly, they can be effectively reformulated by linear constraints without $\bbxi$. 
For the semi-infinite equality constraint \eqref{eq:LDR_h2}, it reduces to a finite linear one, as stated in the following proposition. % based on \cite{kuhn2011primal}.
\begin{proposition} \label{eq:prop1}
Under (AS\ref{as:1}), constraint \eqref{eq:LDR_h2} is equivalent to:
\begin{align}
\mathbf{A} \bm{Y}_{f} = \boldsymbol{\gamma}\mathbf{e}^{\mathsf T}-\mathbf{F}. \label{eq:equlity_1}
\end{align}
\end{proposition} 

\begin{proof}
Recalling $\mathbf{A}\bm{f} = \bm{g} - \bm{d}$ [cf. \eqref{eq:OTS_f}], we can rewrite \eqref{eq:LDR_h2} as $(\mathbf{A} \bm{Y}_{f} -\boldsymbol{\gamma}\mathbf{e}^{\mathsf T}+\mathbf{F})\boldsymbol{\xi} = \mathbf 0,\; \forall {\boldsymbol\xi} \in \boldsymbol{\Xi}$. This implies that the linear hull of $\boldsymbol{\Xi}$ should belong to the null space of the linear operator $(\mathbf{A} \bm{Y}_{f} - \boldsymbol{\gamma}\mathbf{e}^{\mathsf T} + \mathbf{F})$. As $\boldsymbol{\Xi}$ spans the whole sample space under (AS\ref{as:1}), the associated null space is empty and \eqref{eq:equlity_1} holds accordingly. % equivalently we have $\mathbf{A} \bm{Y}_{f} = \boldsymbol{\gamma}\mathbf{e}^{\mathsf T}-\mathbf{F}$. 
$\qed$
\end{proof}
For the inequality constraints \eqref{eq:LDR_i2} and \eqref{eq:LDR_j2},  a well-known equivalence result in robust optimization \cite{kuhn2011primal} leads to a tractable constraint system, as described in the following proposition.
\begin{proposition} \label{eq:prop2}
Under (AS\ref{as:1}), the constraints \eqref{eq:LDR_i2} and \eqref{eq:LDR_j2} are respectively equivalent to:
\begin{subequations}  \label{eq:inequlity}
\begin{align}
   & \bm{Y}_{f} = \mathbf{K}\bm{Y}_{\theta} + \boldsymbol{\Phi}_{1}^{\mathsf T} \mathbf{U},\; \mathbf{K}\boldsymbol{\theta}  + {\mathbf{M}}\circ(\mathbf{e} - \bm{z})-\boldsymbol{\Phi}_{1}^{\mathsf T}\mathbf{t} \geq \bm{f} \label{eq:inequlity_1}\\
   & \bm{Y}_{f} = \mathbf{K}\bm{Y}_{\theta} - \boldsymbol{\Phi}_{2}^{\mathsf T}\mathbf{U},\;  {\mathbf{M}}\circ(\mathbf{e} - \bm{z})-\boldsymbol{\Phi}_{2}^{\mathsf T}\mathbf{t} + \bm{f} \geq \mathbf{K}\boldsymbol{\theta} \label{eq:inequlity_2}
\end{align}
\end{subequations}
where the matrices $\boldsymbol{\Phi}_{1}, \boldsymbol{\Phi}_{2}  \in \mathbb{R}_+^{L \times W}$ % \geq \bm{0}$ %and $\boldsymbol{\Phi}_{2} \in \mathbb{R}^{L \times W} \geq \bm{0}$ 
collect the dual variables for  constraints in \eqref{eq:LDR_i2} and \eqref{eq:LDR_j2}, respectively.
\end{proposition} 
% \begin{proposition} \label{eq:prop2}
% For any $\bm{h} \in \mathbb{R}^{K}$ and ${c} \in \mathbb{R}$, inequality constraint $\bm{h}^{\mathsf T} \boldsymbol{\xi} + c \geq 0, \forall \boldsymbol{\xi} \in \boldsymbol{\Xi}$ is equivalent to:
% \begin{align}
% \exists \: \boldsymbol\varphi  \geq \bm{0}, \: \text{with} -\mathbf{U}^{\mathsf T} \boldsymbol\varphi = \bm{h}  \: \: \text{and} \: \: {\mathbf{t}}^{\mathsf T} \boldsymbol\varphi \leq c. \nonumber 
% \end{align}
% \end{proposition} 
\begin{proof}
This proposition can be viewed as a special case of \cite[Thm.~3.2]{ben2004adjustable}. For any constraint of the form $\bm{h}^{\mathsf T} \boldsymbol{\xi} + m \geq 0, \forall \boldsymbol{\xi} \in \boldsymbol{\Xi}$, under (AS\ref{as:1}) it is equivalent to $0 \leq \min_{ \boldsymbol{\xi}} \{\bm{h}^{\mathsf T} \boldsymbol{\xi} + m: \mathbf{U}\boldsymbol{\xi} \leq \mathbf{t}\}$. The right-hand side expression is essentially a linear program, for which the equivalent dual problem under Slater's conditions becomes  $0 \leq \max_{\varphi} \{-\mathbf{t}^{\mathsf T}\boldsymbol\varphi + m: \boldsymbol\varphi  \geq \bm{0}, -\mathbf{U}^{\mathsf T} \boldsymbol\varphi = \bm{h}\}$, where $\boldsymbol\varphi$ is the vector of dual variables. 
For the maximum of the dual problem to be non-negative, the dual vector
$\boldsymbol\varphi  \geq \bm{0}$ has to satisfy $\mathbf{U}^{\mathsf T} \boldsymbol\varphi = -\bm{h}$ {and} ${\mathbf{t}}^{\mathsf T} \boldsymbol\varphi \leq m$. Thus, constraints \eqref{eq:LDR_i2} and \eqref{eq:LDR_j2} are rewritten into \eqref{eq:inequlity} using this equivalence. 
$\qed$
\end{proof}
Using Propositions \ref{eq:prop1} and \ref{eq:prop2}, we can convert the remaining constraints in \eqref{eq:LDR_2} to deterministic ones  without $\bbxi$ and accordingly formulate the CC-OTS problem as, follows:
\begin{subequations}\label{eq:LDR}
\begin{align}
\min \quad & \bm{c}^{\mathsf T}\bm{g} + \mathbb{E}[\bm{q}^{\mathsf T}\boldsymbol{\gamma} \mathbf{e}^{\mathsf T} \boldsymbol{\xi}] \label{eq:LDR_a}\\
\textrm{s.t.} \quad & \eqref{eq:LDR_b2}, \eqref{eq:LDR_c2}, \eqref{eq:equlity_1}, \eqref{eq:inequlity_1}, \eqref{eq:inequlity_2}  \label{eq:LDR_b}\\
  &   \mathbb{P}\{\mathbf{a}_{i} (\mathbf{x})^{\mathsf T} \boldsymbol{\xi} \leq {b}_{i}(\mathbf{x})\} \geq 1-\epsilon_{i}, \forall i \in \ccalI.
    \label{eq:cc_general}
\end{align}
\end{subequations}
Here, the chance constraints \eqref{eq:cc_general} guarantee that each inequality in~\eqref{eq:set} holds with a probability of at least $1-\epsilon_{i}$, for a pre-specified tolerance level $\epsilon_{i}$.

\subsection{Benchmark Methods for CC-OTS}
We present two benchmark methods for approximating the chance constraints \eqref{eq:cc_general}, which is the most critical step in solving \eqref{eq:LDR}. These two approximation methods give rise to mixed-integer problems and will be used to numerically compare with the proposed DRCC methods later on.

\subsubsection{Sample Average Approximation (SAA)}
Given independently and identically distributed (i.i.d.) uncertainty samples $\{{\boldsymbol{\xi}}^{j}\}$ with $j \in \mathcal{J} := \{1,\cdots,S\}$, the SAA approach \cite{luedtke2010integer} replaces the CC constraints \eqref{eq:cc_general} with the sample-based empirical distribution $\hat{\mathbb{P}}$ that assigns equal mass to all samples.

Under the empirical distribution, the chance constraint \eqref{eq:cc_general} is equivalent to the system of mixed-integer linear constraints: 
\begin{subequations} \label{eq:ben_a}
\begin{align}
  & w^{j}_{i} \in \{0, 1\}, \quad  \forall i \in \ccalI, \, j \in \mathcal{J}.   \label{eq:ben_a1}\\
  & {b}_{i}(\mathbf{x}) - {\mathbf{a}_{i}}(\mathbf{x})^{\mathsf T} \boldsymbol{\xi}^{j} + {\mathrm M} w^{j}_{i} \geq 0, \quad \forall i, \, j. \label{eq:ben_a2}\\
&  \textstyle  \sum_{j=1}^{S} w^{j}_{i} \leq S \epsilon_{i}, \quad \forall i. \label{eq:ben_a3}
\end{align}
\end{subequations}
with a sufficiently large ${\mathrm M}$. 
The binary decision variable $w^{j}_{i}$ indicates whether the constraint $\mathbf{a}_{i} (\mathbf{x})^{\mathsf T} \boldsymbol{\xi}^{j} \leq {b}_{i}(\mathbf{x})$ holds or not.
If $w^{j}_{i} = 0$, \eqref{eq:ben_a2} is equivalent to $\mathbf{a}_{i} (\mathbf{x})^{\mathsf T} \boldsymbol{\xi}^{j} \leq {b}_{i}(\mathbf{x})$ and the constraint holds; otherwise, \eqref{eq:ben_a2} becomes redundant for a large enough $\mathrm M$. If each $\boldsymbol{\xi}^{j}$ is randomly sampled with an equal probability (i.e., $\mathbb{P}(\boldsymbol{\xi}^{j}) = 1/S$), constraint \eqref{eq:ben_a3} guarantees that the sample-based probability of violation $(\frac{1}{S} \sum_{j=1}^{S} w^{j}_{i})$ is not greater than the threshold $\epsilon_{i}$. 
{Under SAA, the resulting approximation of \eqref{eq:LDR} is an MILP. Notice, however, that the number of constraints can grow quickly with the sample size $S$. Due to this scalability issue, the SAA approach will be mainly used for the small test case in numerical studies.
}

% Notice that the SAA method may not guarantee feasible solutions to problem \eqref{eq:LDR}, especially when the sample size $N$ is small. However, as $N$ increases, the number of binary variables and constraints increases correspondingly, making problem \eqref{eq:ben_a} difficult to solve. 
%Due to the feasibility and tractability issues, we will only use the SAA benchmark method for small case study.

\subsubsection{Gaussian Approximation}
This method assumes that $\bm{\xi}$ is a Gaussian random vector with mean $\boldsymbol\mu$ and covariance $\boldsymbol\Sigma$. 
For the chance constraint \eqref{eq:cc_general} with a typical threshold of $\epsilon_{i} \leq \tfrac{1}{2}$, the Gaussian distribution leads to an equivalent second-order cone (SOC) constraint \cite[Sec. 4.4]{boyd2004convex}. 
% This approach has been popularly used for solving CC-OPF problems (e.g., \cite{lubin2015robust,roald2017chance}) due to the tractability of the resultant SOC program (SOCP). 
To briefly introduce the basic idea of this method, we consider the variance of $\mathbf{a}_{i} (\mathbf{x})^{\mathsf T} \boldsymbol{\xi}$ as denoted by $\sigma^{2}$ and constraint $i \in \ccalI$ in \eqref{eq:cc_general} now becomes
% \begin{align}
%     \mathbb{P}\left(\frac{\mathbf{a}_{i} (\mathbf{x})^{\mathsf T} \boldsymbol{\xi}-\mathbf{a}_{i} (\mathbf{x})^{\mathsf T}\boldsymbol\mu}{\sigma} \leq \frac{{b}_{i}(\mathbf x)-\mathbf{a}_{i} (\mathbf{x})^{\mathsf T}\boldsymbol\mu}{\sigma}\right) \geq 1-\epsilon_{i}.
% \end{align}
\begin{align}
    \mathbb{P}\!\left(\!\frac{\mathbf{a}_{i} (\mathbf{x})^{\mathsf T} \boldsymbol{\xi}-\mathbf{a}_{i} (\mathbf{x})^{\mathsf T}\boldsymbol\mu}{\sigma} \leq \frac{{b}_{i}(\mathbf x)-\mathbf{a}_{i} (\mathbf{x})^{\mathsf T}\boldsymbol\mu}{\sigma}\right)\!\!\geq\! 1-\epsilon_{i}. \label{eq:cc_ga}
\end{align}
With $\Phi^{-1} (\epsilon_{i})$ denoting $\epsilon_{i}$-quantile of the standard normal distribution, \eqref{eq:cc_ga} is equivalent to the following SOC constraint:
% \begin{align}
%     \boldsymbol\mu^{\mathsf T}\mathbf{a}_{i} (\mathbf{x}) - \Phi^{-1} (\epsilon_{i}) \sigma \leq {b}_{i}(\mathbf x).
% \end{align}
% Since $\sigma = \sqrt{ {\mathbf{a}_{i} (\mathbf{x})^{\mathsf T}} \boldsymbol\Sigma \mathbf{a}_{i}(\mathbf{x})}$, this is equivalent to the following SOC constraint:
\begin{align}
    \boldsymbol\mu^{\mathsf T}\mathbf{a}_{i} (\mathbf{x}) + \Phi^{-1} (1-\epsilon_{i}) \big \|\boldsymbol\Sigma^{\frac{1}{2}} \mathbf{a}_{i}(\mathbf{x})\big \|_{2} \leq {b}_{i}(\mathbf x) \label{eq:socp}
\end{align}
%where the vector $\mathbf{a}_{i} (\mathbf{x})^{\mathsf T}$ corresponds to unknown variables in constraints \eqref{eq:LDR_d2} - \eqref{eq:LDR_g2}, more specifically, row vectors of $\boldsymbol{\gamma} \mathbf{e}^{\mathsf T}$, ${\bm{Y}}_{{\theta}}$ and $\bm{Y}_{f}$. Hence, it justifies that \eqref{eq:socp} is an SOCP constraint to the decision variable $\mathbf{x}$. In fact, $\mathbf{a}_{i}(\mathbf{x})$ is mainly used here to be in consistency with the chance constraint \eqref{eq:cc_general}.
Under Gaussian approximation, the CC-OTS problem \eqref{eq:LDR} becomes a mixed-integer SOCP (MISOCP). This method requires no sampling, yet its uncertainty model can be too restrictive for the renewable perturbations in practice.

\subsection{Quantifying the Level of Renewable Curtailment}
\label{sec:quantify}
% {\yuqi
% Upon obtaining the optimal solutions (e.g., $\bm{z}^{*}$, $\boldsymbol{\theta}^{*}$, $\bm{f}^{*}$, $\bm{Y}_{\theta}^{*}$, $\bm{Y}_{f}^{*}$) to the CC-OTS problem \eqref{eq:LDR}, we can further quantify the level of actual renewable curtailment needed. Under given scenario of renewable output, we find the maximum amount that can be integrated without violating the feasibility constraints. Accordingly, the following aims to determine the  curtailment vector $ \bm{g}_{c} \in \mathbb{R}_+^{N}$ for given renewable scenario $\tilde{\boldsymbol{\xi}}$:}

In practice, curtailment of renewable generation is used to avoid oversupply and to maintain constraint satisfaction \cite{ross2015multiobjective}. 
Upon solving any CC-OTS problem with the optimal topology $\bm{z}^*$ and other values (denoted by $^*$),
%$\bm{Y}_\theta^{*}$ and $\bm{Y}_f^{*}$, 
one can apply the Monte Carlo method %\cite{rubinstein2016simulation} 
using a large number of uncertainty scenarios to obtain the average of resultant curtailment values. Specifically, for a given renewable scenario $\tilde{{\boldsymbol\xi}}$, we determine the renewable curtailment vector ${\boldsymbol\xi}_{c} \in \mathbb R_{+}^K$ in order to satisfy all network constraints, as given by
\begin{subequations}
\label{eq:optcurl}
\begin{align}
\min \quad & \mathbf 1^{\mathsf T} {\boldsymbol\xi}_{c} \label{eq:optcurl_a}\\
\textrm{s.t.} \quad 
  &  {\boldsymbol\xi}_{c} \in \mathbb R_{+}^K\\
  & \underline{\boldsymbol{\theta}} \leq \boldsymbol{\theta}^{*} + \bm{Y}_{\theta}^{*}(\tilde{{\boldsymbol\xi}} -{\boldsymbol\xi}_{c})
  %- ({\mathbf{B}^{*}})^{-1} \bm{g}_{c} 
  \leq \overline{\boldsymbol{\theta}} \\
  & \underline{\bm{f}} \circ \bm{z}^{*} \leq \bm{f}^{*} + \bm{Y}_{f}^{*}(\tilde{{\boldsymbol\xi}} - {\boldsymbol\xi}_{c}) %\mathbf{K}^{*}({\mathbf{B}^{*}})^{-1} \bm{g}_{c}  
  \leq \overline{\bm{f}} \circ \bm{z}^{*},
\end{align}
\end{subequations}
%the matrices $\mathbf{B}^{*}$ and $\mathbf{K}^{*}$ respectively represent reduced Bbus and line flow matrices corresponding to the given topology under $\bm{z}^*$. 
%matrix $\mathbf{B}$ is defined as:
%\begin{align}
%    \mathbf{B} = \sum_{(i,j)\in\ccalL} b_{ij} (\mathbf{e}_i-\mathbf{e}_j)(\mathbf{e}_i-\mathbf{e}_j)^{\mathsf T}.
%\end{align}
%Under the dc power flow model, the resulting changes in phase angle/line flow can be affinely represented by ${\mathbf{B}^{*}}^{-1} \bm{g}_{c}$ and $\mathbf{K}^{*}{\mathbf{B}^{*}}^{-1} \bm{g}_{c}$, respectively. 
%The problem aims to integrate as much renewable generation as possible while finding the lowest amount of curtailment under each realization $\tilde{{\boldsymbol\xi}}$. 
where the network constraints in \eqref{eq:set} have been simplified to the linear ones in~\eqref{eq:optcurl} by fixing $\bm{z}^*$. 
Note the the curtailment criterion in \eqref{eq:optcurl_a} is essentially the $L_1$ norm of ${\boldsymbol\xi}_{c}$, which gives rise to an efficient linear program (LP) in \eqref{eq:optcurl}. Other criteria such as $L_2$ norm can be used as well, at possibly increased computation complexity. By determining ${\boldsymbol\xi}_{c}$ using~\eqref{eq:optcurl}, the process for quantifying renewable curtailment 
%can be quantified by calculating the average of total renewable curtailment $\mathbf  e^\mathsf{T}\bm{g}_{c}$. 
%This quantification process involves solving   
boils down to computing the average of  $(\mathbf 1^{\mathsf T} {\boldsymbol\xi}_{c})$ over a large number of renewable scenarios. This process serves to evaluate the impact of OTS solutions in terms of renewable curtailment level, which also applies to the DRCC-OTS solutions to be discussed soon. Note this evaluation is completed offline and does not affect the real-time computation of any OTS solution. 
As more frequent constraint violations naturally lead to higher renewable curtailment, the objective cost attained by \eqref{eq:optcurl} serves as an important criterion to evaluate the robustness performance of CC-OTS solutions, as shown by the numerical results in Section \ref{sec:num}.

\section{Distributionally Robust \\ Chance-Constrained OTS}
\label{sec:drcc}
The distributionally robust optimization (DRO) framework has been recognized as a powerful yet potentially tractable approach to deal with uncertainty in energy systems \cite{wei2015distributionally,lu2018security,li2020confidence,zheng2020adaptive,babaei2020distributionally,xie2017distributionally}.
%cite{zhou2021three}
The DRO framework does not assume a particular probability distribution. Instead, it constructs an \textit{ambiguity set} of plausible distributions that are consistent with the available statistical and structural information on uncertainty. A safe decision is then sought that is feasible to the chance constraints for  all distributions within the ambiguity set. Hence, the framework mitigates data overfitting issues and yields superior performance in  out-of-sample (OOS) tests. 
%enjoys satisfactory computation complexity for CC problems.

To develop the DRO-based OTS formulation, consider the distributionally robust chance constraints (DRCC) for \eqref{eq:set} as 
\begin{align}
    \inf_{\mathbb{P} \in \mathcal{P}} \mathbb{P}\{\mathbf{a}_{i}(\mathbf{x})^{\mathsf T} \boldsymbol{\xi} \leq {b}_{i}(\mathbf{x})\} \geq 1-\epsilon_{i} \: \: \forall i \in \ccalI, \label{eq:drcc}
\end{align}
which require each chance constraint to be satisfied under all  probability distributions $\mathbb P\in\mathcal{P}$. %The ambiguity sets provide partial statistical information to characterize the true probability distribution of uncertainty.
Typical ambiguity sets studied in related DRO-based power system decision-making problems fall into the following categories: i) moment-based ambiguity set \cite{zare2018distributionally,zhao2017distributionally,hassan2020stochastic}, ii) distance-based ambiguity set \cite{guo2018data,duan2018distributionally,poolla2020wasserstein} and iii) structural-based ambiguity set \cite{roald2015security,li2019distributionally}.
We consider the DRCC reformulations using the moment-based ambiguity set (mean and mean absolute deviation) and the distance-based ambiguity set (Wasserstein distance). Both of them are amenable to mixed-integer linear programming reformulations.

\subsection{Mean and Mean Absolute Deviation Ambiguity Set}
%The ambiguity sets provide partial statistical information to characterize the true probability distribution of uncertainty.
The mean and mean absolute deviation (mean-MAD) ambiguity set \cite{hanasusanto2015distributionally} is defined as:
\begin{align} \label{eq:ambiguity_1}
   \mathcal{P}_{1} := \left\{\mathbb{P} \in \mathcal{P}_{0} (\boldsymbol\Xi): \mathbb{E}[\boldsymbol{\xi}] = \boldsymbol{\mu}, \ \mathbb{E}[|\boldsymbol{\xi} - \boldsymbol{\mu}|] \leq \boldsymbol{\sigma}\right\},
\end{align}
which includes all distributions with the mean equal to $\boldsymbol{\mu} \in \mathbb{R}^{K}$ and the mean absolute deviation bounded by $\boldsymbol{\sigma} \in \mathbb{R}_{+}^{K}$.
Note that the absolute value and its inequality are both component-wise. 
This ambiguity set can be extended to impose certain dependence structures (see e.g., \cite[Sec.~5]{wiesemann2014distributionally}).
Each worst-case probability $\inf_{\mathbb{P} \in \mathcal{P}} \mathbb{P}\{\mathbf{a}_{i}(\mathbf{x})^{\mathsf T} \boldsymbol{\xi} \leq {b}_{i}(\mathbf{x})\}$ in~\eqref{eq:drcc} over the ambiguity set $\mathcal P= \mathcal{P}_{1}$ boils down to the following optimization problem:
\begin{subequations} \label{eq:int}
\begin{align}
Z_{\mathcal{P}_{1}} = \inf \quad & \int {\mathbbm{1}}\{{\mathbf{a}_{i} (\mathbf{x})^{\mathsf T}}\boldsymbol{\xi} \leq {b}_{i}(\mathbf{x})\}{v} (\mathrm{d}\boldsymbol{\xi})  \label{eq:int_a}\\
\textrm{s.t.} \quad &  {v}(\cdot) \in \mathcal{M}_{+}, \label{eq:int_b}\\ 
  & \int {v} (\mathrm{d}\boldsymbol{\xi}) = 1,  \label{eq:int_c}\\
  & \int \boldsymbol{\xi} {v} (\mathrm{d}\boldsymbol{\xi}) = \boldsymbol{\mu},  \label{eq:int_d}\\
  & \int |\boldsymbol{\xi} - \boldsymbol{\mu}| {v} (\mathrm{d}\boldsymbol{\xi}) \leq \boldsymbol{\sigma}  \label{eq:int_e},
\end{align}
\end{subequations}
where ${\mathbbm{1}}(\cdot)$ denotes the indicator function  for the inequality constraint, while  $\mathcal{M}_{+}$ defines the set of nonnegative measures. Constraints 
\eqref{eq:int_c} - \eqref{eq:int_e} are essentially the integral forms of~\eqref{eq:ambiguity_1}. 
% {\hao (should we define what type of obj and constraint functions (15) has, for the duality?) 
As the objective and constraint functions are all linear in the unknown measure $v(\cdot)$, the problem \eqref{eq:int} is a convex semi-infinite linear program (SILP). If  the DRCC \eqref{eq:drcc} is feasible under $\ccalP_1$, then we have  $Z_{\mathcal{P}_{1}} \geq 1-\epsilon_{i}$ for constraint $i \in \ccalI$. By denoting $\alpha \in \mathbb{R}$, $\boldsymbol{\beta} \in \mathbb{R}^{K}$, and $\boldsymbol{\kappa} \in \mathbb{R}_{+}^{K}$ as the dual variables of constraints \eqref{eq:int_c}-\eqref{eq:int_e}, respectively, %for uncertainty $\bbxi\in\bbXi$ 
we can formulate the dual problem of \eqref{eq:int} as:
\begin{subequations} \label{eq:ind}
\begin{align}
\sup \quad & \alpha + \bbbeta^{\mathsf T} \bbmu - \bbkappa^{\mathsf T} \boldsymbol{\sigma} \label{eq:ind_a}\\
\textrm{s.t.} \quad & \alpha \in \mathbb{R}, \;\bbbeta \in \mathbb{R}^{K}, \;\bbkappa \in \mathbb{R}_{+}^{K}, \label{eq:ind_b}\\
  & {\mathbbm{1}}\{\mathbf{a}_{i} (\mathbf{x})^{\mathsf T}\boldsymbol{\xi} \leq {b}_{i}(\mathbf{x})\} \geq \alpha + \bbbeta^{\mathsf T} \bbxi - \bbkappa^{\mathsf T} |\bbxi - \bbmu| \;\; \forall {\boldsymbol\xi} \in  \boldsymbol{\Xi}. \label{eq:ind_c}
\end{align}
\end{subequations}
Strong duality holds as the ambiguity set $\mathcal{P}_{1}$ satisfies the Slater's condition \cite{hanasusanto2017ambiguous} for the SILP \eqref{eq:int}. 
The semi-infinite constraint for the dual problem \eqref{eq:ind} boils down to two cases according to the indicator ${\mathbbm{1}}(\cdot)$ in \eqref{eq:ind_c}. Specifically, it equals to 0 for any $\bbxi$ such that $\mathbf{a}_{i} (\mathbf{x})^{\mathsf T}\boldsymbol{\xi} > {b}_{i}(\mathbf{x})$, or 1 for any other choice of $\bbxi$. These two cases can be reformulated using standard convex duality theory to arrive at the %dualized again and jointly constitute the 
following equivalent linear constraints:
\begin{subequations}\label{eq:L}
\begin{align}
  & \alpha' + \bbbeta'^{\mathsf T} \bbmu - \bbkappa'^{\mathsf T} \boldsymbol{\sigma} \geq (1-\epsilon_{i})\lambda' \label{eq:LL_a}\\
  & \alpha' + ({\bbpi'}_{1}^{\mathsf T} - {\bbtau'}_{1}^{\mathsf T})\bbmu + {\bbpsi'}_{1}^{\mathsf T} \mathbf{t}  \leq \lambda' \\
  & {\bbbeta'}^{\mathsf T} + {\bbtau'}_{1}^{\mathsf T} = {\bbpi'}_{1}^{\mathsf T} + {\bbpsi'}_{1}^{\mathsf T} \mathbf{U}  \\
  & {\bbpi'}_{1}^{\mathsf T} + {\bbtau'}_{1}^{\mathsf T} = {\bbkappa'}^{\mathsf T}\\
  & \alpha' + ({\bbpi'}_{2}^{\mathsf T} - {\bbtau'}_{2}^{\mathsf T})\bbmu  + {\bbpsi'}_{2}^{\mathsf T} \mathbf{t} \leq {b}_{i}(\mathbf{x}) \label{eq:LL_e}\\
  & {\bbbeta'}^{\mathsf T} + \mathbf{a}_{i} (\mathbf{x})^{\mathsf T} + {\bbtau'}_{2}^{\mathsf T} = {\bbpi'}_{2}^{\mathsf T}  + {\bbpsi'}_{2}^{\mathsf T} \mathbf{U} \label{eq:LL_f}\\
  & {\bbpi'}_{2}^{\mathsf T} + {\bbtau'}_{2}^{\mathsf T} = {\bbkappa'}^{\mathsf T}. \label{eq:LL_g}
\end{align}
\end{subequations}
The dual variables $\bbpi' \in \mathbb{R}_{+}^{K}$ and $\bbtau' \in \mathbb{R}_{+}^{K}$  are introduced for the epigraph based constraints $\boldsymbol{\rho} \geq \boldsymbol{\xi} - \boldsymbol{\mu}$ and $\boldsymbol{\rho} \leq \boldsymbol{\xi} - \boldsymbol{\mu}$, respectively; the dual variables $\bbpsi' \in \mathbb{R}_{+}^{W}$ are assigned to the linear constraints of the support set
$\bbXi$. Note that the dual variable $\lambda > 0$ corresponding to the new constraint $\mathbf{a}_{i} (\mathbf{x})^{\mathsf T}\boldsymbol{\xi} > {b}_{i}(\mathbf{x})$ introduces bilinearity in the original decision variables $\mathbf x$. We address this by dividing all constraints with $\lambda$ and performing the change of variables for the %has been inversely scaled to deal with the resultant bi-linear issue.
%Specifically, the 
primal variables  in \eqref{eq:ind}  as  $\alpha' = \frac{\alpha}{\lambda} \in \mathbb{R}$ (similarly for $\bbbeta'$ and $\bbkappa'$), and the dual variables as $\lambda' = \frac{1}{\lambda} \in \mathbb{R}_{+}$ (similarly for the aforementioned $\bbpi'$ and $\bbtau'$).

% where $\alpha'\in \mathbb{R}, \: \bbbeta'\in \mathbb{R}^{K}, \: \bbkappa'\in \mathbb{R}_{+}^{K}, \: \lambda'\in \mathbb{R}_{+}, \: \bbpi'\in \mathbb{R}_{+}^{K}, \: \bbtau' \in \mathbb{R}_{+}^{K}, \: \bbpsi' \in \mathbb{R}_{+}^{W}$ are defined as decision variables in the optimization problem. 
% Notice that $\alpha'\in \mathbb{R}, \: \bbbeta'\in \mathbb{R}^{K}, \: \bbkappa'\in \mathbb{R}_{+}^{K}$ are redefined variables which are different from those in \eqref{eq:ind}. 

% {\hao (i'm not sure how to interpret these variables and how they are diff.)}

\begin{proposition} \label{eq:prop3}
The DRCC-OTS problem under ambiguity set $\mathcal{P}_{1}$ is equivalent to the following optimization problem:
\begin{subequations}\label{eq:DRCC1}
\begin{align}
\min \quad & \bm{c}^{\mathsf T}\bm{g} + \mathbb{E}[\bm{q}^{\mathsf T}\boldsymbol{\gamma} \mathbf{e}^{\mathsf T} \boldsymbol{\xi}] \\
\textrm{s.t.} \quad & \alpha'\in \mathbb{R}, \: \bbbeta'\in \mathbb{R}^{K}, \: \bbkappa'\in \mathbb{R}_{+}^{K} \\
  &  \lambda'\in \mathbb{R}_{+}, \: \bbpi'\in \mathbb{R}_{+}^{K}, \: \bbtau' \in \mathbb{R}_{+}^{K}, \: \bbpsi' \in \mathbb{R}_{+}^{W}\\
  &   \eqref{eq:LDR_b}, \: \eqref{eq:LL_a} - \eqref{eq:LL_g}.
\end{align}
\end{subequations}
\end{proposition} 
Thanks to all the linear constraints, the DRCC-OTS problem in \eqref{eq:DRCC1} is an MILP. The DRCC-OTS significantly improves the scalability over the SAA-based MILP problem, as it effectively uses dualization techniques to attain a fixed set of linear constraints such that \eqref{eq:drcc} holds for any distribution in $\ccalP_1$.  %suffers from scalability issues, this DRCC MILP reformulation is more tractable. Although during the reformulation we have introduced a few more decision variables as in \eqref{eq:DRCC1_b} - \eqref{eq:DRCC1_c}, 
Therefore, the resulting problem \eqref{eq:DRCC1} is \textit{scenario-free} and of low computational complexity for efficient implementations in real time.

\noindent \textbf{Incorporating Multimodality Information:} To obtain less conservative solutions to the aforementioned DRCC model, one can further incorporate multimodality information of the uncertainty into the formulation. This additional structural information is particularly relevant to the problem studied in the paper, as it has been observed that wind energy data exhibits multimodal behavior \cite{hu2016estimating,yurucsen2016probability,wang2021wind}. 
To this end, we assume the actual distribution to be a mixture of $m$ distinct distributions $\mathbb{P}_{1}, \cdots, \mathbb{P}_{m}$, with known probabilities $p_{1}, \cdots, p_{m}$, and each $\mathbb{P}_{j}$ has known mean and MAD values $(\boldsymbol{\mu}_{j}, \boldsymbol{\sigma}_{j})$. In this setting, the ambiguity set with multimodality information is given by
\begin{align} 
   \mathcal{P}_{1}' := \sum_{j=1}^{m} p_{j} \mathcal{P}_1(\boldsymbol{\mu}_{j}, \boldsymbol{\sigma}_{j}),
\end{align}
where $\mathcal{P}_1(\boldsymbol{\mu}_{j}, \boldsymbol{\sigma}_{j})$ denotes the mean-MAD ambiguity set \eqref{eq:ambiguity_1} with  mean $\boldsymbol{\mu}_{j}$ and MAD $\boldsymbol{\sigma}_{j}$. 
Each worst-case probability $\inf_{\mathbb{P} \in \mathcal{P}_1'} \mathbb{P}\{\mathbf{a}_{i}(\mathbf{x})^{\mathsf T} \boldsymbol{\xi} \leq {b}_{i}(\mathbf{x})\}$ can be cast as the following problem:
\begin{subequations} \label{eq:int_multi}
\begin{align}
\!Z_{\mathcal{P}_{1}'} \!= \!\inf\! \quad & \sum_{j=1}^{m} p_{j} \int {\mathbbm{1}}\{{\mathbf{a}_{i} (\mathbf{x})^{\mathsf T}}\boldsymbol{\xi} \leq {b}_{i}(\mathbf{x})\}{v_j} (\mathrm{d}\boldsymbol{\xi})  \label{eq:int_a}\\
\textrm{s.t.} \quad &  {v_j} (\cdot) \in \mathcal{M}_{+}, \quad \forall j = 1, \cdots, m, \label{eq:int_b}\\ 
  & \int {v_j} (\mathrm{d}\boldsymbol{\xi}) = 1, \quad \forall j = 1, \cdots, m, \label{eq:int_c}\\
  & \int \boldsymbol{\xi} {v_j} (\mathrm{d}\boldsymbol{\xi}) = \boldsymbol{\mu}_{j}, \quad \forall j = 1, \cdots, m,  \label{eq:int_d}\\
  & \int |\boldsymbol{\xi} - \boldsymbol{\mu}_{j}| {v_j} (\mathrm{d}\boldsymbol{\xi}) \leq \boldsymbol{\sigma}_{j}, \quad \forall j = 1, \cdots, m  \label{eq:int_e}.
\end{align}
\end{subequations}
However, applying similar derivations as in \eqref{eq:int}-\eqref{eq:L} to the multimodal model leads to a non-convex problem, as there will be multiple bilinear products $\lambda_{j} \mathbf{a}_{i} (\mathbf{x})$ and $\lambda_{j} {b}_{i}(\mathbf{x}), \forall j = 1, \cdots, m$ for the different modes that cannot be handled simultaneously.
To deal with this bilinearity, we propose to adopt the block coordinate descent (BCD) algorithm in \cite[Sec.~5]{hanasusanto2015distributionally} for a tractable solution, as described in Algorithm \ref{alg:alg1}. 

For ease of exposition, we denote all the dual variables that are not directly coupled with $\mathbf{x}$ as $\Upsilon_{j} = \{\alpha_{j},\bbbeta_{j},\bbkappa_{j},{\bbpi}_{j},{\bbtau}_{j},{\bbpsi}_{j}\}, \forall j = 1, \cdots, m$.
The BCD algorithm starts with an initial solution $\mathbf{x}^{0}$, which can be the optimal solution from the unimodality model \eqref{eq:DRCC1}. Per iteration $t$, an uncertainty quantification problem is solved to find the worst-case probability under the previous iterate $\mathbf{x}^{t-1}$. Once the dual multipliers $\lambda_{j}, \Upsilon_{j}$ are obtained, we fix $\lambda_{j}$ and solve the DRCC-OTS to update the iterate $\mathbf{x}^{t}$. This iterative approach is repeated until the difference between consecutive objectives is below a prescribed convergence threshold $\omega$. Note that the optimization problems solved in each iteration are convex and can be solved efficiently. We will investigate this multimodality model in numerical tests, as well.

%of the objective function $|f^{t} - f^{t-1}|$ is below the convergence threshold $\omega$.

\begin{figure}[!t] 
 \removelatexerror
  \begin{algorithm}[H]
 \caption{Block Coordinate Descent Algorithm}
 \label{alg:alg1}
 \begin{algorithmic}[1]
 \renewcommand{\algorithmicrequire}{\textbf{Input:}}
 \renewcommand{\algorithmicensure}{\textbf{Output:}}
 \REQUIRE $\mathbf{x}^{0}, p_{j}, \boldsymbol\Xi_{j}, \boldsymbol{\mu}_{j}, \boldsymbol{\sigma}_{j}, \forall j = 1, \cdots, m$
 \ENSURE  $\mathbf{x}$
 \\ \textit{Initialization}: \text{Initial feasible solution $\mathbf{x}^{0}$} 
  \STATE Get objective value $f^{0}$ using $\mathbf{x}^{0}$, set $t=1$.
% \\ \textit{LOOP Process}
  \FOR {$t = 1$ to $t_{\max}$}
  \STATE Uncertainty Quantification: Find the optimal $\left(\Upsilon_{j}^{*}, \lambda_{j}^{*} \right)$ to the dual of \eqref{eq:int_multi} with input $\mathbf x^{t-1}$. % $\sup\{\sum_{j=1}^{m} (\alpha_{j} + \bbbeta_{j}^{\mathsf T} \bbmu_{j} - \bbkappa_{j}^{\mathsf T} \boldsymbol{\sigma}_{j}) \}$.
   Set $\lambda_{j}^{t} \leftarrow \lambda_{j}^{*}$.
  \STATE Policy Update: Fix $\lambda_{j}^{t}$, solve DRCC-OTS and obtain the optimal  $(\mathbf{x}^{*},\Upsilon_{j}^{*})$.
   Set $\mathbf{x}^{t} \leftarrow \mathbf{x}^{*}$ and compute the  objective value $f^{t}$.
  \IF {($|f^{t} - f^{t-1}| < \omega$)}
  \STATE Stop the algorithm, and set $\mathbf{x} = \mathbf{x}^{t}$
  \ENDIF
  %\ENDIF |f^{t} - f^{t-1}| < \omega
  \ENDFOR
 \RETURN $\mathbf{x}$
 \end{algorithmic}
 \end{algorithm}
\end{figure}

\subsection{Wasserstein Ambiguity Set}
% The Wasserstein distance has become a popular metric in DRO \cite{hanasusanto2015distributionally,esfahani2018data,zhang2019optimal} in recent years. {\hao (looks like new refs again. diff from [30-33]??)}
%Different from the aforementioned moment-based ambiguity set, the Wasserstein ambiguity set contains all distributions that are close to a nominal distribution. 
The DRCC with the ambiguity set described in this section ensures the robustness against all probability distributions within a prescribed Wasserstein distance from the empirical distribution~$\hat{\mathbb{P}}$. Compared with the mean-MAD criterion, the Wasserstein metric is purely data-driven and constructed using actual data samples. 
With more samples available, the latter better reveals the actual uncertainty distribution and thus can lead to less conservative DRCC solutions.

We adopt the $\infty$-Wasserstein ambiguity set which is known to enjoy a more tractable reformulation \cite{xie2019distributionally,xie2021distributionally}.
% Suppose $\{\boldsymbol{\xi}^{j}\}_{j \in \mathcal{J}}$ are independently and identically distributed (i.i.d.) samples from the unknown uncertainty distribution, which yield a discrete empirical distribution $\hat{\mathbb{P}}$.
The $\infty$-Wasserstein ambiguity set is defined as
\begin{align}
   \mathcal{P}_{2} := \left\{\mathbb{P} 
   \in \mathcal{P}_{0} (\boldsymbol\Xi): d_{\infty}(\mathbb{P},\hat{\mathbb{P}}) \leq {\delta} \right\} \label{eq:P2},
\end{align}
where $\delta > 0$ is a given Wasserstein radius %which can affect 
that determines the finite-sample performance guarantee of the DRCC problem; see e.g., \cite{bertsimas2021two}.
The radius parameter $\delta$ depends on the number of sample $S$ in a monotonically decreasing fashion. One choice of setting $\delta$ is $\delta = \eta S^{-\frac{1}{k\max\{K,2\}}}$ \cite[Cor.~1]{bertsimas2021two}, where $K$ is the dimension of the uncertainty while $k$ and $\eta$ are problem-dependent constants.
% where $k \in \{{\underline{\dim}}(\boldsymbol\Xi), \cdots, \dim(\boldsymbol\Xi)\}$ and $\eta > 0$ is a constant.
%:
%\begin{align}
%    \delta = \eta S^{-\frac{1}{k\max\{K,2\}}}
%\end{align}
% \begin{align}
%     \delta = B \sqrt{\frac{2}{N} \log(\frac{1}{1-\beta})}
% \end{align}
% where $B$ is the diameter of supporting space, $N$ is the number of data and $\beta$ is the confidence level. One can also utilize the technique proposed in \cite{duan2018distributionally} to improve the radius in case of conservatism issues.
The $\infty$-Wasserstein distance between two distributions $\mathbb{P}_{1}$ and $\mathbb{P}_{2}$ is given by
\begin{align}
    d_{\infty}(\mathbb{P}_{1},\mathbb{P}_{2}) := \inf \quad & \esssup\|{\tilde{\bm{\xi}}_{1} - \tilde{\bm{\xi}}_{2}}\| \nonumber\\
    \textrm{s.t.} \quad & \mathbb{P} \in \mathcal{P}_{0} (\mathbb R^{K} \times \mathbb R^{K}), \label{eq:was} %\nonumber\\
%    & \mathbb{P}: \text{joint distribution of } \tilde{\bm{\xi}}_{1} \text{ and } \tilde{\bm{\xi}}_{2} \nonumber
\end{align}
where $\mathbb{P}$ is the joint distribution of $\tilde{\bm{\xi}}_{1}$ and $\tilde{\bm{\xi}}_{2}$ with marginals $\mathbb{P}_{1}$ and $\mathbb{P}_{2}$, respectively. 
We use $\esssup$ to denote the essential supremum of a function and $\|\cdot\|$ to denote a norm in $\mathbb{R}^{K}$.
% {\hao the  $L_{\infty}$-norm based distance} with respect to the joint distribution $\mathbb{P}$.}
For each constraint $i \in \ccalI$, suppose a big-M coefficient
${\mathrm M}_{j}$ exists to bound
\begin{align*}
{\mathrm M}_{j} \geq \max_{\mathbf{x}} \big\{ {\mathbf{a}_{i} (\mathbf{x})^{\mathsf T}}{\boldsymbol{\xi}}^{j} + {\delta} \|\mathbf{a}_{i}(\mathbf{x}) \|_{*} -  {b}_{i}(\mathbf{x}) \big \},~\forall j\in \ccalJ 
\end{align*}
where $\left\| \, \cdot \, \right\|_{*}$ is the corresponding dual norm. This way, the DRCC in \eqref{eq:drcc} under $\ccalP_2$ can be represented as the following mixed-integer  constraints \cite[Cor.~4]{xie2019distributionally}:
\begin{subequations}\label{eq:Wasser2}
\begin{align}
  & \eqref{eq:ben_a1}, \eqref{eq:ben_a3}  \label{eq:was1}\\
  & {\delta} \|\mathbf{a}_{i} (\mathbf x)\|_{*} \leq  {b}_{i}(\mathbf{x}) - {\mathbf{a}_{i} (\mathbf{x})^{\mathsf T}}{\boldsymbol{\xi}}^{j} + {\mathrm M}_{j}w^{j}_{i}, \quad \forall i, \, j. \label{eq:was2}
\end{align}
\end{subequations}
Note that this reformulation mimics the SAA-based one in \eqref{eq:ben_a}, by changing the lower bound of the right-hand side (RHS) of \eqref{eq:was2} from $0$ to ${\delta} \|\mathbf{a}_{i} (\mathbf x)\|_{*}$, which acts as a regularizer. Intuitively, a smaller radius ${\delta}$ implies the restriction to distributions more similar to the empirical one $\hat{\mathbb{P}}$. 
{Accordingly, the constraint \eqref{eq:was2} becomes less restrictive.}
As ${\delta}$ decreases to $0$, $\ccalP_2$ reduces to the singleton %boils down to 
$\hat{\mathbb{P}}$ itself and \eqref{eq:Wasser2} becomes equivalent to the SAA approach.

\begin{proposition} \label{eq:prop4}
The DRCC-OTS problem under ambiguity set $\mathcal{P}_{2}$ is equivalent to the following optimization problem:
\begin{subequations}\label{eq:DRCC2}
\begin{align}
\min \quad & \bm{c}^{\mathsf T}\bm{g} + \mathbb{E}[\bm{q}^{\mathsf T}\boldsymbol{\gamma} \mathbf{e}^{\mathsf T} \boldsymbol{\xi}] \label{eq:DRCC2_a}\\
\textrm{s.t.} \quad &  \eqref{eq:LDR_b}, \: \eqref{eq:was1}, \: \eqref{eq:was2}.
    \label{eq:DRCC2_b}
\end{align}
\end{subequations}
\end{proposition} 
For better numerical tractability, we have picked the $L_{\infty}$-norm as the ground metric for Wasserstein distance in \eqref{eq:was}, for which the dual norm is $L_1$ in \eqref{eq:was2} and the problem \eqref{eq:DRCC2} becomes an MILP.
% To simplify the problem \eqref{eq:DRCC2}, we have chosen the $L_{\infty}$ norm for the Wasserstein distance in \eqref{eq:was}. {\hao This is because its dual norm is $L_1$  in \eqref{eq:was2} and thus} the  problem \eqref{eq:DRCC2} becomes an MILP.
Due to the similarity to SAA, the DRCC-OTS under the Wasserstein metric also incurs the same complexity issue as the number of constraints grows with sample size $|\ccalJ|$. Nonetheless, the choice of $\infty$-Wasserstein ambiguity set already improves the tractability over the traditional Wasserstein metric as in \cite{guo2018data,duan2018distributionally}. 
Compared to the mean-MAD ambiguity set,  the DRCC-OTS problem under the Wasserstein ambiguity set takes more computation time especially for large systems, but its data-driven feature makes the resulting solutions less conservative with sufficient number of data samples.

\begin{remark}[distributionally robust objective]
We can also extend the DRCC formulations to include a distributionally robust objective function. To achieve this, the term $\mathbb{E}[\bm{q}^{\mathsf T}\boldsymbol{\gamma} \mathbf{e}^{\mathsf T} \boldsymbol{\xi}]$ in the objective functions of \eqref{eq:DRCC1} and \eqref{eq:DRCC2} can be changed to $ \sup_{\mathbb{P} \in \mathcal{P}} \mathbb{E}[\bm{q}^{\mathsf T}\boldsymbol{\gamma} \mathbf{e}^{\mathsf T} \boldsymbol{\xi}]$. The latter is equivalent to $\bm{q}^{\mathsf T}\boldsymbol{\gamma} \sup_{\mathbb{P} \in \mathcal{P}} \mathbb{E}[\mathbf{e}^{\mathsf T} \boldsymbol{\xi}]$, in which $\sup_{\mathbb{P} \in \mathcal{P}}\mathbb{E}[\mathbf{e}^{\mathsf T} \boldsymbol{\xi}]$ can be determined for a given ambiguity set $\mathcal{P}$ similar to the steps for analyzing the constraints. Note that this change only affects  the coefficient for scaling the term  $\bm{q}^{\mathsf T}\boldsymbol{\gamma}$. Thus, for simplicity, this work did not incorporate a DR cost objective, as $\mathbb{E}[\mathbf{e}^{\mathsf T} \boldsymbol{\xi}]$ can be viewed as a lower bound for $\sup_{\mathbb{P} \in \mathcal{P}}\mathbb{E}[\mathbf{e}^{\mathsf T} \boldsymbol{\xi}]$.
\end{remark}

\section{Numerical Results}
\label{sec:num}

\begin{table*}[t] 
  \centering
  \caption{Performance of Approaches A1-A4 in the IEEE 14-bus System}
\begin{tabular}{ |P{2.4cm}||P{0.9cm}|P{1.3cm}|P{1.3cm}|P{1.3cm}|P{1.3cm}|P{1.3cm}|P{1.3cm}|P{1.3cm}|P{1.3cm}| }
 \hline
 & & \multicolumn{2}{c|}{A1 (CC-SAA)}& \multicolumn{2}{c|}{A2 (CC-Gaussian)}& \multicolumn{2}{c|}{A3 (DRCC-MAD)}& \multicolumn{2}{c|}{A4 (DRCC-Wasserstein)} \\
 \hline
  & $1-\epsilon$ & 0.95 & 0.90 & 0.95 & 0.90 & 0.95 & 0.90 & 0.95 & 0.90 \\
 \hline
\multirow{3}{8em}{Switching \\Decision} 
 & $L_{o} = 1$  & [16] & [16] & [16] & [16] & [16] & [16] & [16] & [16]\\
 & $L_{o} = 2$  & [9;18] & [9;18] & [9;18] & [9;20] & [9;20] & [9;20] & [9;20] & [9;20]\\
 & $L_{o} = 3$  & [9;18;19] & [9;18;19] & [9;19;20] & [9;12;20] & [9;12;18] & [9;12;18] & [9;12;18] & [9;12;18]\\
 \hline
 \multirow{3}{8em}{Run Time (sec)} 
 & $L_{o} = 1$  & 3.07 & 2.93 & 0.51 & 0.24 & 0.33 & 0.38 & 6.87 & 5.67\\
 & $L_{o} = 2$  & 1.73 & 2.29 & 0.25 & 0.43 & 0.30 & 0.38 & 11.52 & 9.05\\
 & $L_{o} = 3$  & 1.54 & 1.70 & 0.25 & 0.21 & 0.40 & 0.29 & 6.79 & 8.08\\
 \hline
 \multirow{3}{8em}{Out-of-sample Costs (\textdollar/$h$)}  
 & $L_{o} = 1$  & 545.34 & 542.10 & 546.01 & 544.48 & 556.02 & 552.90 & 549.95 & 547.02\\
 & $L_{o} = 2$  & 518.95 & 516.57 & 519.76 & 518.36 & 529.91 & 526.28 & 525.36 & 524.76\\
 & $L_{o} = 3$  & 515.89 & 512.50 & 518.70 & 517.49 & 524.77 & 521.37 & 518.55 & 517.81\\
 \hline
  \multirow{3}{8em}{Average Violation Rates} 
 & $L_{o} = 1$  & 0.0616 & 0.0896 & 0.0402 & 0.0412 & 0.0220 & 0.0232 & 0.0258 & 0.0278\\
 & $L_{o} = 2$  & 0.1054 & 0.1058 & 0.0472 & 0.0478 & 0.0238 & 0.0264 & 0.0294 & 0.0298\\
 & $L_{o} = 3$  & 0.1070 & 0.1184 & 0.0648 & 0.0676 & 0.0292 & 0.0304 & 0.0358 & 0.0364\\
 \hline
 \multirow{3}{8em}{Renewable Curtailment} 
 & $L_{o} = 1$  & 2.4957 & 2.5637 & 2.4327 & 2.4381 & 2.2253 & 2.3218 & 2.0611 & 2.1042\\
 & $L_{o} = 2$  & 2.6626 & 2.7561 & 2.4350 & 2.5124 & 2.2671 & 2.3062 & 1.9863 & 2.3505\\
 & $L_{o} = 3$  & 2.5136 & 2.5815 & 2.4156 & 2.4268 & 2.0670 & 2.1367 & 2.0272 & 2.2341\\
 \hline
\end{tabular}
  \label{tab:1}
\end{table*}

In this section, we present the numerical results validating the proposed DRCC-OTS  methods using the IEEE 14-bus and 118-bus test cases. Other benchmark approaches are implemented too for performance comparisons in terms of robustness. For ease of exposition, we refer to all the tested approaches  as the following:
\begin{itemize}
  \item A1: Sample-average approximation benchmark (MILP)
  \item A2: Gaussian approximation benchmark (MISOCP)
  \item A3: DRCC under mean-MAD ambiguity set (MILP)
  \item A4: DRCC under $\infty$-Wasserstein ambiguity set (MILP)
\end{itemize}
% In addition, a brief comparison of these approaches is provided in the following table:
% \begin{table}[h!] 
% \caption{Comparison of Chance Constrained Transmission Switching Approaches} \label{tab:sd_14} 
% \centering 
% {\renewcommand{\arraystretch}{1.2}
% \begin{tabular}{|p{1.8cm}||p{1.2cm}|p{1.2cm}|p{1.2cm}|p{1.2cm}|}
% \hline 
%  & A1 & A2 & A3 & A4 \\ 
% \hline 
% Approximation & Yes  & Yes & No & No\\  
% Sample Based & Yes  & No & No & Yes\\ 
% Program & MILP & MISCOP & MILP & MILP \\   
% \hline  
% \end{tabular}}
% \end{table}
% Note that the additional benchmark A5 for the  two-stage problem \eqref{eq:LDR_2} is included to demonstrate the improvements using CC approaches. As detailed in  \cite{zhou2020transmission}, it can be solved by dualizing constraints \eqref{eq:LDR_d2}-\eqref{eq:LDR_j2} over the support set $\boldsymbol\Xi$ to produce a scenario-free MILP. In this sense, its complexity order is on par with A3 with a slight reduction due to fewer variables/constraints in A5. 
% As A5 incorporates all scenarios in $\boldsymbol{\Xi}$ at no violation of constraints, it provides an upper bound of the total cost to all CC approaches. 
%{\yuqi The problem can be reformulated as an MILP, which is also scenario-free. It shares similar complexity as DRCC under mean-MAD ambiguity set, however, we can expect it to have slightly better computational efficiency as fewer variables are defined.}

We use the hourly wind power data from the ERCOT market~\cite{ERCOT} from 2018 to 2020 by scaling it according to the size and load demand of the test systems. 
Due to the seasonality of wind patterns, its uncertainty may vary over the year. 
% \cite{keane2010capacity}. 
Thus, we have used data samples from all four seasons to build the ambiguity set and scenarios.
To better compare the performance, we conduct out-of-sample (OOS) experiments by partitioning the dataset into training and  testing samples. The optimal solutions are produced using the in-sample training dataset, while the costs and constraint violations are evaluated on the OOS testing dataset.
The 14-bus case is used to test and compare all the CC approaches (A1, A2, A3, and A4). 
Due to the tractability issues of the scenario-based approaches, for the larger 118-bus case we mainly evaluate (A2) and (A3), with a robust benchmark of $\epsilon=0$ (i.e., zero violation), all of which are scenario-free. Each individual CC tolerance level $\epsilon_{i}$ has been set to be the same value $\epsilon$ for simplicity.
The test case parameters are obtained from MATPOWER, and the OTS problems (MILP, MISOCP) are solved using Gurobi. The solver was set to utilize up to 12 available threads with a solution tolerance of $1\mathrm{e}^{-2}$. All the numerical tests have been implemented on a regular laptop equipped with Intel\textsuperscript{\textregistered} CPU @ 2.60 GHz and 16 GB of RAM using the MATLAB\textsuperscript{\textregistered} R2020b simulator.

% In sample tests (CC_SAA)
% Out-of-sample tests (OOS_SAA_formal)

\subsection{{IEEE} 14-Bus System Tests}
\label{sec:num_14}
The original IEEE 14-bus system consists of 20 lines and 5 conventional generators. We add 3 wind farms to the case, located at buses 3, 6, and 13, respectively.
Given that the marginal gain reduces with more lines to switch, we have used a maximum of $L_{o} = 3$ opening lines.
For the sample-based approaches (A1 and A4), increasing the sample size can lead to more accurate results at the cost of increased problem dimension and computation time. Therefore, we have used $S=200$ samples for both A1 and A4. 
The Wasserstein radius $\delta$ is selected according to \cite[Cor.~1]{bertsimas2021two} and tuned to comply with solutions from other approaches (A1, A2, and A3).
% L0 = 0, the renewable curtailment in the table from left to right:
% 2.4936 2.7535; 2.9238 2.9265; 2.0299 2.1996; 1.7852 1.9970
By setting $1-\epsilon = 0.95$ or $1-\epsilon = 0.90$, we compare the optimal switching decisions, run times, OOS costs and constraint violation rates for A1-A4. The results are listed in Table~\ref{tab:1}. To evaluate the OOS testing performance, we used 5,000 random samples from the actual wind data and recorded the percentage of violated constraints by averaging over all testing samples.
% {\hao (why mention the following? do you  want to directly say that you have on purpose increased congestion and how?)}
% Notice that in some cases, it is possible that varying chance constraint threshold $\epsilon$ within certain range does not necessarily affect the optimal solution (e.g., the constraints are non-binding). Such issues can be resolved by either modifying $\epsilon$ value strategically or adding line congestions \cite{zhang2016distributionally}.
To avoid cases where the majority of line flow constraints are non-binding under uncertainty, we have slightly adjusted the line flow limits to increase the transmission congestion level as in \cite{zhang2016distributionally}.
 
The switching decisions tend to vary among the four approaches when $L_{o} = 2$ or 3. Interestingly, the switching decisions largely remain the same as the tolerance $\epsilon$ changes except for A2. Note that the tolerance $\epsilon$ more significantly affects the other decisions, namely the generation dispatch $\bm{g}$ and AGC coefficients $\bbgamma$. This becomes clear when comparing the OOS costs, as discussed shortly. By and large, the run times of all approaches are very reasonable. Sample-based approaches (A1 and A4) take more time, while the scenario-free ones (A2 and A3) are much faster (within 1 second). For the sample-based A4, the $\infty$-Wasserstein metric makes its run time comparable to A1, while offering better DRO guarantees. 

In terms of OOS performance, the DRCC approaches (A3 and A4) incur slightly higher total costs than the other two. This is expected as the DRCC approaches are designed to account for a variety of distributions in the ambiguity set. Between A3 and A4, the Wasserstein metric has lower OOS costs as its solutions are more data-driven and less conservative, as mentioned earlier. Note that although A3 and A4 produce exactly the same switching decisions, their OOS costs still differ due to their differences in the $\bm{g}$ and $\bbgamma$ decisions. This difference can also be observed for all approaches with $L_o=1$.  Under fixed $(1-\epsilon)$, the OOS costs generally are reduced as $L_{o}$ increases, and a smaller $(1-\epsilon)$ allows for more violations of constraints and thus lowers the total costs.

\begin{figure}[t!]
\centering
\vspace{-2pt}
\includegraphics[trim=0cm 0cm 2cm 0.5cm,clip=true,totalheight=0.32\textheight]{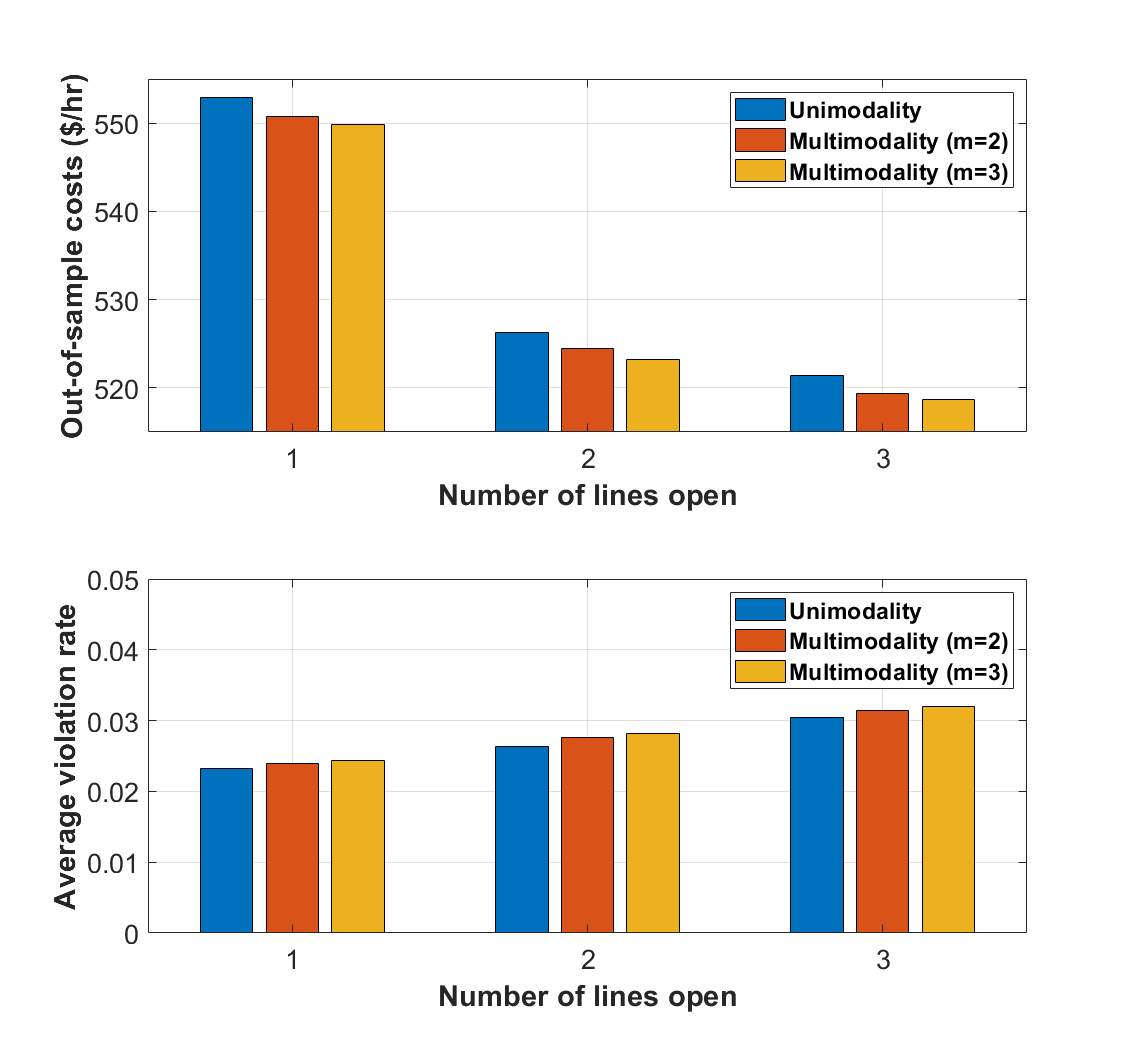}
\caption{Comparisons of OOS costs and average violation rate for unimodality and multimodality models in A3 approach.}
\label{fig:14_b}
\end{figure}

The comparisons on constraint violation and renewable curtailment in OOS testing are very important for evaluating the robustness performance.
Ideally, the OOS violation rates should not exceed the pre-specified threshold $\epsilon$. However, this is rarely the case for A1, because the SAA design relies on the approximation by the empirical distribution and cannot strictly enforce the robustness guarantees. In addition, A2 has one instance of exceeding $\epsilon=0.05$ for the case of $L_o =3$, which speaks to its disadvantage of solely relying on the assumption on Gaussian distributed uncertainty. Compared to A1 and A2, the proposed A3 and A4 have nicely maintained very low constraint violation rates for all choices of $L_o$, thanks to their DRCC based design principle. This is especially important for a smaller value of $\epsilon$, where the robustness guarantees are more difficult to enforce. 
Using the renewable curtailment quantification approach in Sec.~\ref{sec:quantify}, we have shown the clear improvement of DRCC approaches (A3 and A4) over CC approaches in reducing curtailment levels. As the former has demonstrated proved robustness guarantees, grid congestion is less likely to occur and so is the renewable curtailment. 
Fig.~\ref{fig:14_b} further shows the trade-off between OOS costs and average violation rates attained by the multimodality model based mean-MAD approaches with $1-\epsilon = 0.90$. Compared with the unimodality benchmark (A3), we increase the number of modes to be $m=2$ or 3. We observe that including the multimodality information leads to a less conservative DRCC solution with decreasing OOS costs. Meanwhile, the average violation rates slightly increase with $m$ as a trade-off.

In summary, the proposed DRCC approaches demonstrate a graceful trade-off between the total cost and constraint satisfaction rate. They can reliably limit the occurrence of constraint violations and thus reduce the level of renewable curtailment, at some incremental cost.

\begin{figure}[t!]
\centering
    \centering
    \subfloat[]
    {
        \includegraphics[trim=0cm 0cm 0cm 1cm,clip=true,totalheight=0.20\textheight]{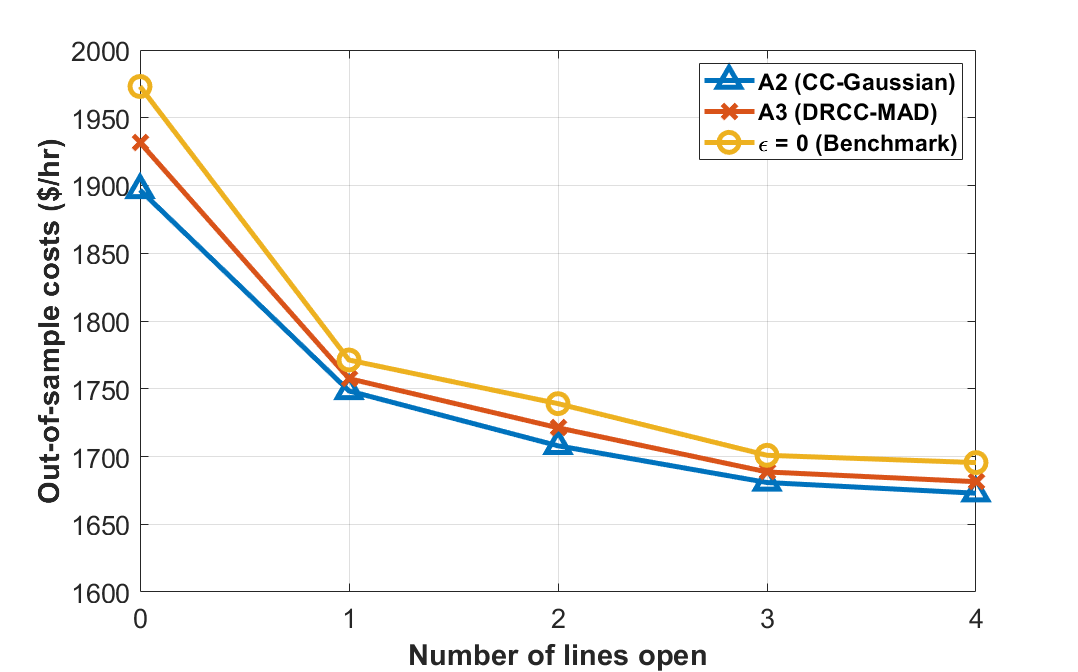}
     \label{fig:118_a}

    }
    \\
    \centering
    \subfloat[]
    {
        \includegraphics[trim=0cm 0cm 0cm 1cm,clip=true,totalheight=0.20\textheight]{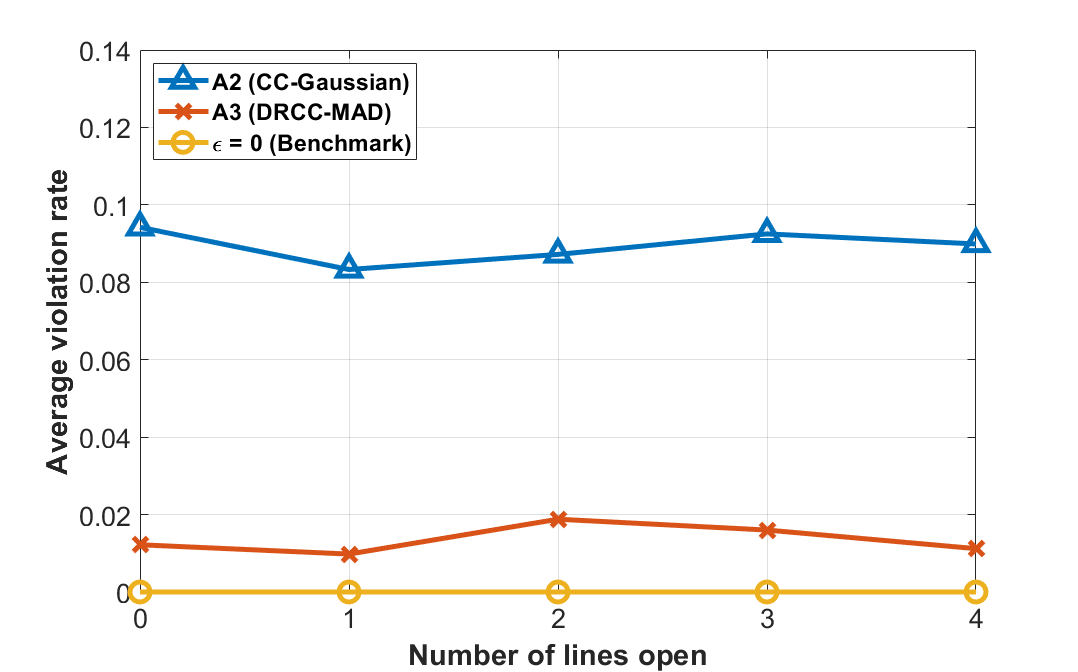}
        \label{fig:118_d}
    }
     \\
    \centering
    \subfloat[]
    {
        \includegraphics[trim=0cm 0cm 0cm 1cm,clip=true,totalheight=0.20\textheight]{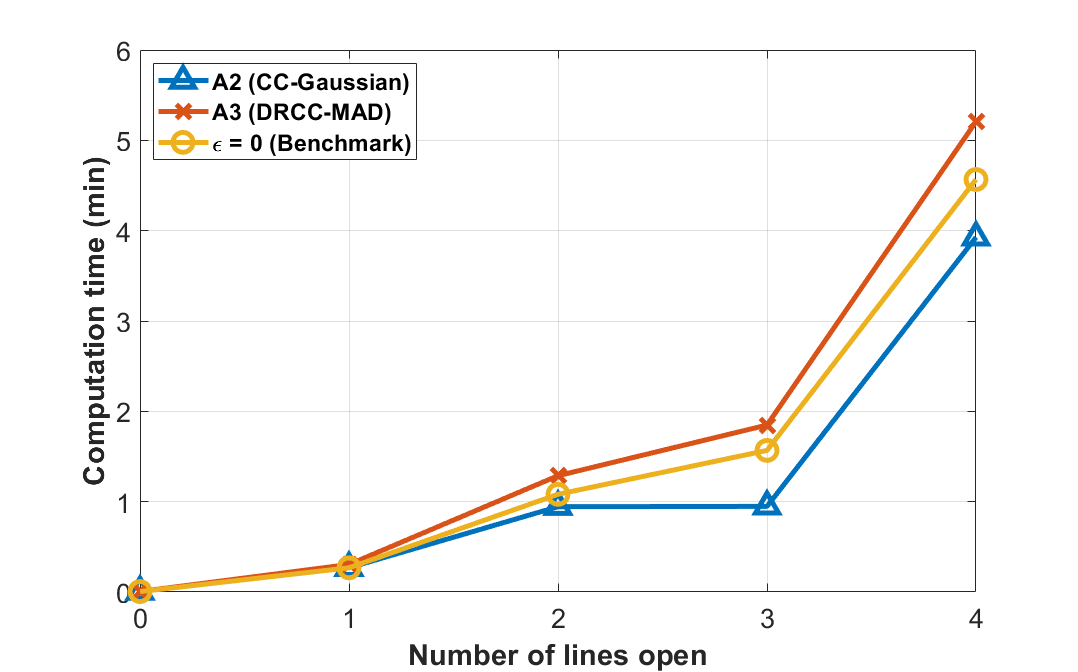}
        \label{fig:118_b}
    }
    \caption{Comparisons of the (a) OOS costs; (b) average violation rates; and (c) computation time for the 118-bus system.}
    \label{fig:sample_subfigures}
    \vspace*{-3.5mm}
\end{figure}

\subsection{{IEEE} 118-Bus System Tests}
\label{sec:num_118}
We have also tested the approaches on the IEEE 118-bus system, consisting of 186 lines and 19 conventional generators. Five wind farms have been added, which are located at buses 10, 23, 57, 62 and 86, respectively. Chance constraints have been applied on half of the line flow constraints with the other half strictly enforced. 
The sample-based methods (A1 and A4) are extremely inefficient for the mixed-integer formulation, especially for large systems (as high as 5-10 hours).
Due to this scalability issue, we have only compared the scenario-free approaches A2 and A3, with a robust benchmark by setting the tolerance level $\epsilon=0$.

First, we use the CC tolerance $1-\epsilon = 95\%$ to compare the OOS performance under different $L_o$, as plotted in Fig.~\ref{fig:118_a}. 
% Notice that the reduction on the generation costs becomes less significant as $L_{o}$ increases, as also mentioned in previous studies (see e.g., \cite{fisher2008optimal,qiu2015chance}).
Overall, the OOS costs increase slightly from A2 to A3, and both are smaller than the benchmark cost. This trend is consistent with the average rate of constraint violations as shown in Fig.~\ref{fig:118_d}. 
Compared with the benchmark, A3 achieves 1.0\% cost reduction on average, while A2 achieves 1.8\%.
Notably, the constraint violation rates for the proposed A3 are nicely maintained around $0.02$ which is smaller than the threshold $\epsilon=0.05$, while those for A2 can go up to roughly $0.09$ that exceeds the tolerance level. This large-system test again confirms the aforementioned improvement of the proposed DRCC-OTS approaches over A2 in terms of guaranteed constraint satisfaction. 
%The DRCC approach A3 achieves reasonable violation rates at around $1\%$ - $2\%$, while robust approach A5 achieves zero violation rates. In comparison, the Gaussian method A2 has the worst performance in the larger 118-bus system, which gives the highest violation rates among these approaches.
% \begin{figure}[t!]
% \centering
% \vspace{-2pt}
% \includegraphics[trim=0cm 0cm 1cm 0cm,clip=true,totalheight=0.24\textheight]{figure_118_a.png}
% \caption{Comparison of out-of-sample costs in the 118-bus system.}
% \label{fig:118_a}
% \end{figure}
% \begin{figure}[t!]
% \centering
% \vspace{-2pt}
% \includegraphics[trim=0cm 0cm 1cm 0cm,clip=true,totalheight=0.24\textheight]{figure_118_d.png}
% \caption{Average violation rates for A2, A3 and A5 in the 118-bus system.}
% \label{fig:118_d}
% \end{figure}
% \begin{table}[t!] 
% \caption{Average Violation Rates for A2, A3 and A5 in the 118-bus System} 
% \centering 
% {\renewcommand{\arraystretch}{1.2}
% \begin{tabular}{ |c| c | c | c | c | c | c | } 
% \hline 
%  $L_{o}$ & 0 & 1 & 2 & 3 & 4 & 5\\
% \hline 
% A2 & 0.0942 & 0.0833 & 0.0872 & 0.0925 & 0.0899 & 0.0908\\  
% A3 & 0.0122 & 0.0098 & 0.0188 & 0.0160 & 0.0112 & 0.0170 \\
% A5 & 0 & 0 & 0 & 0 & 0 & 0 \\
% \hline  
% \end{tabular} }
% \label{tab:2}
% \end{table}
Moreover, we have compared the average run time, as shown in Fig.~\ref{fig:118_b}. In general, the run time of the proposed A3 is on par with the other two, with a moderate increase for larger $L_{o}$ values.  
% In addition, the optimal switching decisions of these approaches are listed in Table \ref{tab:3}. Some critical lines (e.g., lines 96, 118, 131, and 149) have been commonly selected by different approaches. When $L_{o} \leq 2$, the switching decisions are the same for all approaches. As $L_{o} \geq 3$, the decisions of A2 and benchmark are very similar while A3 has picked a few different lines for better CC guarantees. %Notice that even the switching decisions are the same in some cases, the OOS costs can be still different as shown in Fig. \ref{fig:118_a}(a). This is because the decisions on generation dispatch and AGC coefficients can be different.
Lastly, we compare the OOS costs of the proposed A3 for different tolerance levels, by varying $\epsilon$ in the range of $0\%-30\%$, as shown in Fig.~\ref{fig:118_c}.
With fixed $L_{o}$, a larger $\epsilon$ value leads to gradually decreasing costs, by allowing higher occurrences of constraint violations. If we compare to the OOS costs of the benchmark approach ($\epsilon = 0$), the proposed DRCC-based A3 can  attain lower costs with a roughly $1.1\%$ reduction on average. Notice that the marginal gain of cost reduction is minimal at higher tolerance levels ($\epsilon$ increasing from 20\% to 30\%).  Generally speaking, the range of $\left[5\%, 20\%\right]$ is deemed appropriate for $\epsilon$ in practical operations  \cite{zhang2016distributionally,duan2018distributionally}.

% \begin{figure}[t!]
% \centering
% \vspace{-2pt}
% \includegraphics[trim=0cm 0cm 1cm 0cm,clip=true,totalheight=0.24\textheight]{figure_118_b.png}
% \caption{Comparison of computation time between A2 and A3 for the 118-bus system.}
% \label{fig:118_b}
% \end{figure}

% $6$ & [51;98;118;131;135;149] & [41;97;118;131;135;149] & [41;96;118;131;135;149] \\ 

% \begin{table}[t!] 
% \caption{Switching Decisions for Chance Constrained Transmission Switching} 
% \centering 
% {\renewcommand{\arraystretch}{1.2}
% \begin{tabular}{ |c| c | c | c | c | c | } 
% \hline 
%  $L_{o}$ & $1$ & $2$ & $3$ & $4$ & $5$\\
% \hline 
% A3 & 131 & 128,130 & 131,135,149 & 54,128,131,156 & 97,118,131,135,149  \\  
% A2 & 128 & 131,150 & 96,131 & 131 & 131 \\
% A5 & 131 & 11,131 & 96,131,150 & 131,1,1,1 & 131,1,2,3,4 \\ 

% \hline  
% \end{tabular} }
% \label{tab:3}
% \end{table} 

% \begin{table}[t!] 
% \caption{Switching Decisions for the 118-bus System} 
% \centering 
% {\renewcommand{\arraystretch}{1.2}
% \begin{tabular}{ |c| c | c | c | } 
% \hline 
%  $L_{o}$ & A2 (CC-Gaussian) & A3 (DRCC-MAD) & $\epsilon = 0$ (Benchmark)\\
% \hline 
% $1$ & 131 & 131 & 131 \\  
% $2$ & 96;131 & 96;131 & 96;131 \\
% $3$ & 96;131;150 & 96;131;156 & 96;131;150 \\   
% $4$ & 96;117;131;156 & 96;118;125;131 & 96;131;135;149 \\
% $5$ & 96;118;131;135;149 & 96;118;131;136;149 & 96;118;131;135;149 \\ 
% \hline  
% \end{tabular} }
% \label{tab:3}
% \end{table} 

\begin{figure}[t!]
\centering
\vspace{-2pt}
\includegraphics[trim=0cm 0cm 1cm 0cm,clip=true,totalheight=0.22\textheight]{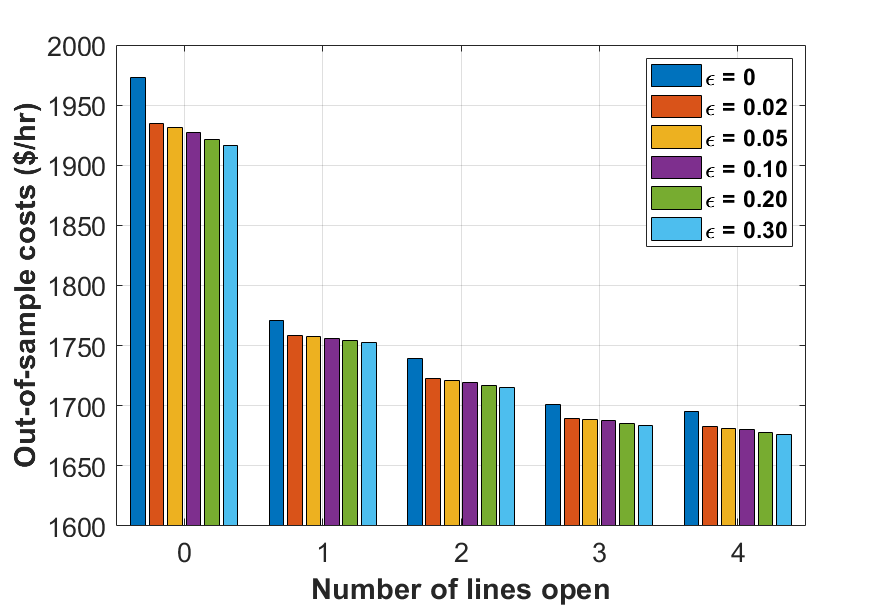}
\caption{Comparisons of the OOS costs attained by A3 under different tolerance levels and $L_o$ values.}
\label{fig:118_c}
\end{figure}

In summary, the proposed DRCC approaches can effectively limit the occurrences of violating line flow constraints by accounting for the distributional ambiguity of  uncertainty.  In particular, the mean-MAD ambiguity criterion leads to a scenario-free, tractable MILP reformulation, with comparable complexity to the CC and benchmark approaches. 
%The numerical results have corroborated the performance and computational efficiency of the DRCC approach for achieving real-time OTS operations under uncertainty.

\section{Conclusions} \label{sec:con}

This paper considered the chance-constrained optimal transmission switching (CC-OTS) problem to account for renewable uncertainty in power systems. We proposed to simplify the two-stage OTS problem by establishing the equivalence of linear decision rules (LDR) based reformulation. Due to the lack of distributional knowledge on the uncertainty, we pursued a distributionally robust chance-constrained (DRCC) OTS paradigm that can ensure the guarantees over an ambiguity set of uncertainty distributions. Both moment-based and distance-based ambiguity sets have been considered, leading to scalable MILP problems through dualization. Numerical tests validated  the performance improvements  of the proposed DRCC approaches over the CC alternatives in terms of guaranteed constraint violation rates.  Between the two proposed DRCC-OTS approaches, the one using the mean-MAD ambiguity set brought lower computation complexity on par with other scenario-free approaches, while the one using the Wasserstein ambiguity led to less conservative solutions by adapting to the actual data samples. Future work includes reducing the complexity of scenario-based DRCC-OTS by simplifying the Wasserstein ambiguity set and developing
machine learning approaches for accelerated OTS computations in real-time.

% {\hao Future work missing?}
%The reformulations and trade-off analysis among different approaches are presented.
%Numerical studies on the IEEE 14-bus and 118-bus systems demonstrate the validity and performance of the proposed distributionaly robust approaches for efficient CC-OTS implementations.

%%%%%%%%%%%%%%%%%%%%%%%%%%%%%%%%%%%%%%%%%%%%%%%%%%%%%%%%%%%%%%%%%%%%%%
%                                                                    %
%       Bibliography %
%
%%%%%%%%%%%%%%%%%%%%%%%%%%%%%%%%%%%%%%%%%%%%%%%%%%%%%%%%%%%%%%%%%%%%%%

\bibliography{bibliography}

\appendix
Here we will present the detailed steps to derive the equivalent reformulation from \eqref{eq:ind} to \eqref{eq:L} for the problem \eqref{eq:DRCC1}. 
% The problem \eqref{eq:ind} is the following maximization problem:
% \begin{subequations} 
% \begin{align}
% \sup \quad & \alpha + \bbbeta^{\mathsf T} \bbmu - \bbkappa^{\mathsf T} \boldsymbol{\sigma} \label{eq:ind_a}\\
% \textrm{s.t.} \quad & \alpha \in \mathbb{R}, \;\bbbeta \in \mathbb{R}^{K}, \;\bbkappa \in \mathbb{R}_{+}^{K}, \label{eq:ind_b}\\
%   & {\mathbbm{1}}\{\mathbf{a}_{i} (\mathbf{x})^{\mathsf T}\boldsymbol{\xi} \leq {b}_{i}(\mathbf{x})\} \geq \alpha + \bbbeta^{\mathsf T} \bbxi - \bbkappa^{\mathsf T} |\bbxi - \bbmu| \;\; \forall {\boldsymbol\xi} \in  \boldsymbol{\Xi}. \label{eq:ind_c}
% \end{align}
% \end{subequations}
Using the definition of the indicator function ${\mathbbm{1}}(\cdot)$, the semi-infinite constraint in problem \eqref{eq:ind} boils down to the following two cases:
\begin{subequations} \label{eq:ind2}
\begin{align}
  & \bbalpha + \bbbeta^{\mathsf T} \bbxi - \bbkappa^{\mathsf T} |\bbxi - \bbmu| \leq 1, \quad \forall \bbxi \label{eq:ind2a}\\
  & \bbalpha + \bbbeta^{\mathsf T} \bbxi - \bbkappa^{\mathsf T} |\bbxi - \bbmu| \leq 0, \quad \forall \bbxi: \mathbf{a}_{i} (\mathbf{x})^{\mathsf T}\boldsymbol{\xi} > {b}_{i}(\mathbf{x}) \label{eq:ind2b}
\end{align}
\end{subequations}
Specifically, the right hand side equals to 0 for any $\bbxi$ such that $\mathbf{a}_{i} (\mathbf{x})^{\mathsf T}\boldsymbol{\xi} > {b}_{i}(\mathbf{x})$, or 1 for any other choice of $\bbxi$. These two cases can be reformulated using standard convex duality theory \cite{hanasusanto2015distributionally}.
Specifically, \eqref{eq:ind2a} is equivalent to the following:
\begin{subequations} \label{eq:ind3}
\begin{align}
\sup \quad & \bbalpha + \bbbeta^{\mathsf T} \bbxi - \bbkappa^{\mathsf T} \bbrho \leq 1 \label{eq:ind3_a}\\
\textrm{s.t.} \quad & \bbxi \in \mathbb{R}^{K}   \label{eq:ind3_b}\\
  & \bbrho \geq \bbxi - \bbmu \label{eq:ind3_d} \quad \quad \qquad \; \; \; \; \; ({\bbpi_1})\\
  & \bbrho \geq \bbmu - \bbxi \label{eq:ind3_e} \quad \quad \qquad \; \; \; \; \; ({\bbtau_1})
\end{align}
\end{subequations}
Dualizing this optimization problem implies that there exists non-negative dual variables $\bbpi_1 \in \mathbb{R}_{+}^{K}, \bbtau_1 \in \mathbb{R}_{+}^{K}$ such that 
\begin{align}
\bbalpha - \bbkappa^{\mathsf T}\bbrho + \bbpi_{1}^{\mathsf T}(\bbrho + \bbmu)& + \bbtau_{1}^{\mathsf T}(\bbrho - \bbmu) - 1 \nonumber\\
& \leq \min_{\mathbf{U}\boldsymbol{\xi} \leq \mathbf{t}} \;  (-\bbbeta^{\mathsf T} + \bbpi_{1}^{\mathsf T} - \bbtau_{1}^{\mathsf T})\bbxi
\end{align}
We can dualize the right hand side again using the uncertainty support, and it leads to the following equivalent constraints:
\begin{subequations}
\begin{align}
  & \alpha + ({\bbpi}_{1}^{\mathsf T} - {\bbtau}_{1}^{\mathsf T})\bbmu + {\bbpsi}_{1}^{\mathsf T} \mathbf{t}  \leq 1 \\
  & {\bbbeta}^{\mathsf T} + {\bbtau}_{1}^{\mathsf T} = {\bbpi}_{1}^{\mathsf T} + {\bbpsi}_{1}^{\mathsf T} \mathbf{U}  \\
  & {\bbpi}_{1}^{\mathsf T} + {\bbtau}_{1}^{\mathsf T} = {\bbkappa}^{\mathsf T}
\end{align}
\end{subequations}
where $\bbpsi_{1} \in \mathbb{R}_{+}^{W}$ are introduced as the dual variables for the linear constraints ($\mathbf{U}\boldsymbol{\xi} \leq \mathbf{t}$) for the support set
$\bbXi$.
Similarly, we can derive the equivalent constraints for \eqref{eq:ind2b}. The constraints \eqref{eq:ind2b} are equivalent to the following:
\begin{subequations} \label{eq:ind5}
\begin{align}
\sup \quad & \bbalpha + \bbbeta^{\mathsf T} \bbxi - \bbkappa^{\mathsf T} \bbrho \leq 0 \label{eq:ind5_a}\\
\textrm{s.t.} \quad & \bbxi \in \mathbb{R}^{K}   \label{eq:ind5_b}\\
  & \mathbf{a}_{i} (\mathbf{x})^{\mathsf T}\boldsymbol{\xi} > {b}_{i}(\mathbf{x}) \label{eq:ind5_c} \quad \qquad  \:  ({\lambda})\\
  & \bbrho \geq \bbxi - \bbmu \label{eq:ind5_d} \quad \quad \qquad \; \; \; \; \; ({\bbpi_2})\\
  & \bbrho \geq \bbmu - \bbxi \label{eq:ind5_e} \quad \quad \qquad \; \; \; \; \; ({\bbtau_2})
\end{align}
\end{subequations}
Dualizing it implies that there exists non-negative dual variables $\lambda \in \mathbb{R}_{+}, \bbpi_2 \in \mathbb{R}_{+}^{K}, \bbtau_2 \in \mathbb{R}_{+}^{K}$ such that 
\begin{align}
\bbalpha - \bbkappa^{\mathsf T}\bbrho + & \bbpi_{2}^{\mathsf T}(\bbrho + \bbmu) +  \bbtau_{2}^{\mathsf T}(\bbrho - \bbmu) - \lambda b_{i}(\mathbf x)
\nonumber\\
&\leq \min_{\mathbf{U}\bbxi \leq \mathbf{t}} \;  (-\bbbeta^{\mathsf T} + \bbpi_{2}^{\mathsf T} - \bbtau_{2}^{\mathsf T} - \lambda \mathbf{a}_{i} (\mathbf{x})^{\mathsf T})\bbxi
\end{align}
We can dualize the right hand size again, which leads to the following equivalent constraints:
\begin{subequations}
\begin{align}
  & \alpha + ({\bbpi}_{2}^{\mathsf T} - {\bbtau}_{2}^{\mathsf T})\bbmu + {\bbpsi}_{2}^{\mathsf T} \mathbf{t}  \leq \lambda b_{i}(\mathbf x) \\
  & {\bbbeta}^{\mathsf T} + \lambda \mathbf{a}_{i} (\mathbf{x})^{\mathsf T} + {\bbtau}_{2}^{\mathsf T} = {\bbpi}_{2}^{\mathsf T} + {\bbpsi}_{2}^{\mathsf T} \mathbf{U}  \\
  & {\bbpi}_{2}^{\mathsf T} + {\bbtau}_{2}^{\mathsf T} = {\bbkappa}^{\mathsf T}
\end{align}
\end{subequations}
where $\bbpsi_{2} \in \mathbb{R}_{+}^{W}$ are introduced as the dual variables for the linear constraints $\mathbf{U}\boldsymbol{\xi} \leq \mathbf{t}$.
Recall that the objective function \eqref{eq:ind_a} also needs to satisfy:
\begin{align}
    \alpha + \bbbeta^{\mathsf T} \bbmu - \bbkappa^{\mathsf T} \boldsymbol{\sigma} \geq 1- \epsilon_{i} \label{eq:obj}
\end{align}
Therefore, the original problem is equivalent to the following constraints:
\begin{subequations}
\begin{align}
  & \alpha + \bbbeta^{\mathsf T} \bbmu - \bbkappa^{\mathsf T} \boldsymbol{\sigma} \geq 1- \epsilon_{i}\\
  & \alpha + ({\bbpi}_{1}^{\mathsf T} - {\bbtau}_{1}^{\mathsf T})\bbmu + {\bbpsi}_{1}^{\mathsf T} \mathbf{t}  \leq 1 \\
  & {\bbbeta}^{\mathsf T} + {\bbtau}_{1}^{\mathsf T} = {\bbpi}_{1}^{\mathsf T} + {\bbpsi}_{1}^{\mathsf T} \mathbf{U}  \\
  & {\bbpi}_{1}^{\mathsf T} + {\bbtau}_{1}^{\mathsf T} = {\bbkappa}^{\mathsf T}\\
  & \alpha + ({\bbpi}_{2}^{\mathsf T} - {\bbtau}_{2}^{\mathsf T})\bbmu + {\bbpsi}_{2}^{\mathsf T} \mathbf{t}  \leq \lambda b_{i}(\mathbf x) \label{eq:bi1}\\
  & {\bbbeta}^{\mathsf T} + \lambda \mathbf{a}_{i} (\mathbf{x})^{\mathsf T} + {\bbtau}_{2}^{\mathsf T} = {\bbpi}_{2}^{\mathsf T} + {\bbpsi}_{2}^{\mathsf T} \mathbf{U}  \label{eq:bi2}\\
  & {\bbpi}_{2}^{\mathsf T} + {\bbtau}_{2}^{\mathsf T} = {\bbkappa}^{\mathsf T}
\end{align}
\end{subequations}
Notice that the dual variable $\lambda > 0$ corresponding to the constraint $\mathbf{a}_{i} (\mathbf{x})^{\mathsf T}\boldsymbol{\xi} > {b}_{i}(\mathbf{x})$ introduces bilinearity in the above formulation, due to $\lambda b_{i}(\mathbf x)$ in constraint \eqref{eq:bi1} and $\lambda \mathbf{a}_{i} (\mathbf{x})^{\mathsf T}$ in constraint \eqref{eq:bi2}. To address this, we divide all the constraints with $\lambda$ and redefine variables $\alpha' = \frac{\alpha}{\lambda} \in \mathbb{R}$, $\bbbeta' = \frac{\bbbeta}{\lambda} \in \mathbb{R}^{K}$, $\bbkappa' = \frac{\bbkappa}{\lambda} \in \mathbb{R}^{K}_{+}$, $\bbpi' = \frac{\bbpi}{\lambda} \in \mathbb{R}^{K}_{+}$, $\bbtau' = \frac{\bbtau}{\lambda} \in \mathbb{R}^{K}_{+}$, $\lambda' = \frac{1}{\lambda} \in \mathbb{R}_{+}$. Eventually, we arrive at the following equivalent \textit{linear} constraints, as in \eqref{eq:L}:
\begin{subequations}
\begin{align}
  & \alpha' + \bbbeta'^{\mathsf T} \bbmu - \bbkappa'^{\mathsf T} \boldsymbol{\sigma} \geq (1-\epsilon_{i})\lambda' \\
  & \alpha' + ({\bbpi'}_{1}^{\mathsf T} - {\bbtau'}_{1}^{\mathsf T})\bbmu + {\bbpsi'}_{1}^{\mathsf T} \mathbf{t}  \leq \lambda' \\
  & {\bbbeta'}^{\mathsf T} + {\bbtau'}_{1}^{\mathsf T} = {\bbpi'}_{1}^{\mathsf T} + {\bbpsi'}_{1}^{\mathsf T} \mathbf{U}  \\
  & {\bbpi'}_{1}^{\mathsf T} + {\bbtau'}_{1}^{\mathsf T} = {\bbkappa'}^{\mathsf T}\\
  & \alpha' + ({\bbpi'}_{2}^{\mathsf T} - {\bbtau'}_{2}^{\mathsf T})\bbmu  + {\bbpsi'}_{2}^{\mathsf T} \mathbf{t} \leq {b}_{i}(\mathbf{x})\\
  & {\bbbeta'}^{\mathsf T} + \mathbf{a}_{i} (\mathbf{x})^{\mathsf T} + {\bbtau'}_{2}^{\mathsf T} = {\bbpi'}_{2}^{\mathsf T}  + {\bbpsi'}_{2}^{\mathsf T} \mathbf{U}\\
  & {\bbpi'}_{2}^{\mathsf T} + {\bbtau'}_{2}^{\mathsf T} = {\bbkappa'}^{\mathsf T}.
\end{align}
\end{subequations}

%\begin{comment}
%\begin{thebibliography}{1}

%\itemsep 2pt

%\bibitem{ANSI}
%emph{American National Standard
%For Electric Power Systems and Equipment- Voltage Ratings (60 Hertz)}, ANSI C84.1-2011, American National Standards Institute (ANSI), Inc. 

%\end{thebibliography}
%\end{comment}

\bibliographystyle{IEEEtran}

\itemsep2pt

%\section*{Biographies}

\begin{IEEEbiography}[{\includegraphics[width=1in,
height=1.25in,keepaspectratio]{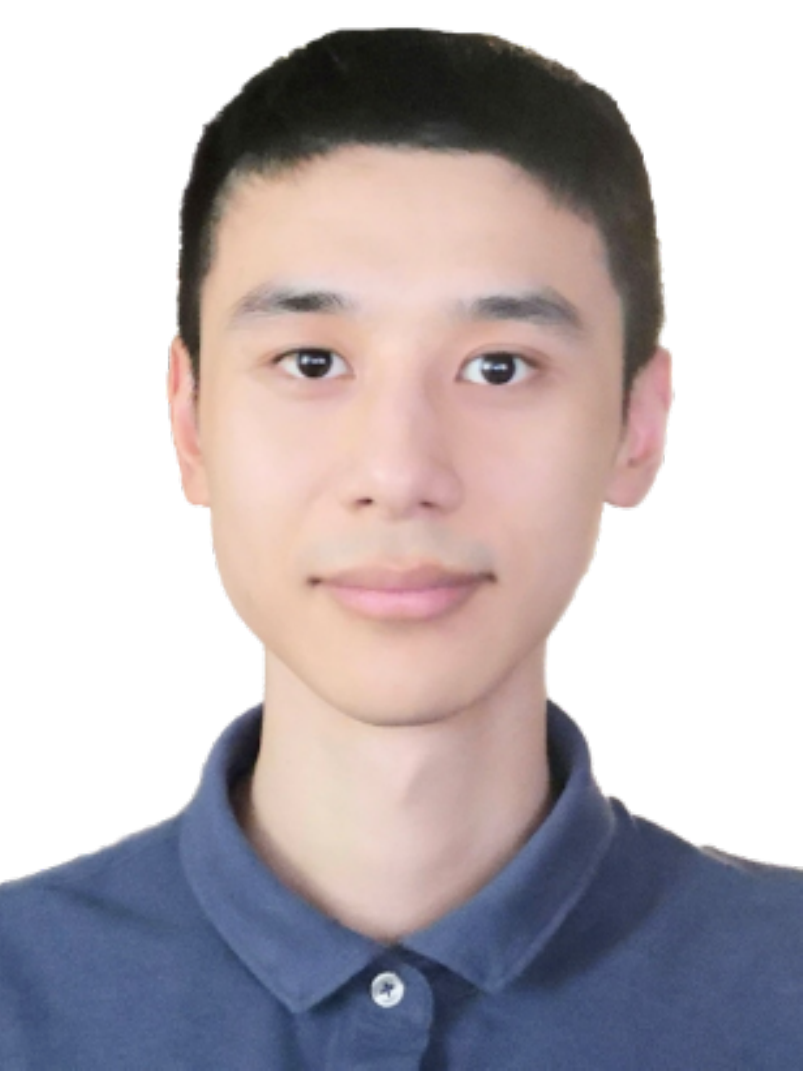}}]{Yuqi Zhou} (S'16) received the B.E. degree in information engineering from Xi'an Jiaotong University, Xi'an, China, in 2015. He received the M.S. degree in electrical engineering from Texas A\&M University, College Station, TX, USA, in 2018.
He is currently working toward the Ph.D. degree at the University of Texas at Austin, Austin, TX, USA.
His current research interests include topology control and optimization in high-voltage transmission systems, and power system operations under uncertainty.
\end{IEEEbiography}

%\vspace{-15cm}
\begin{IEEEbiography}[{\includegraphics[width=1in,
		height=1.25in,keepaspectratio]{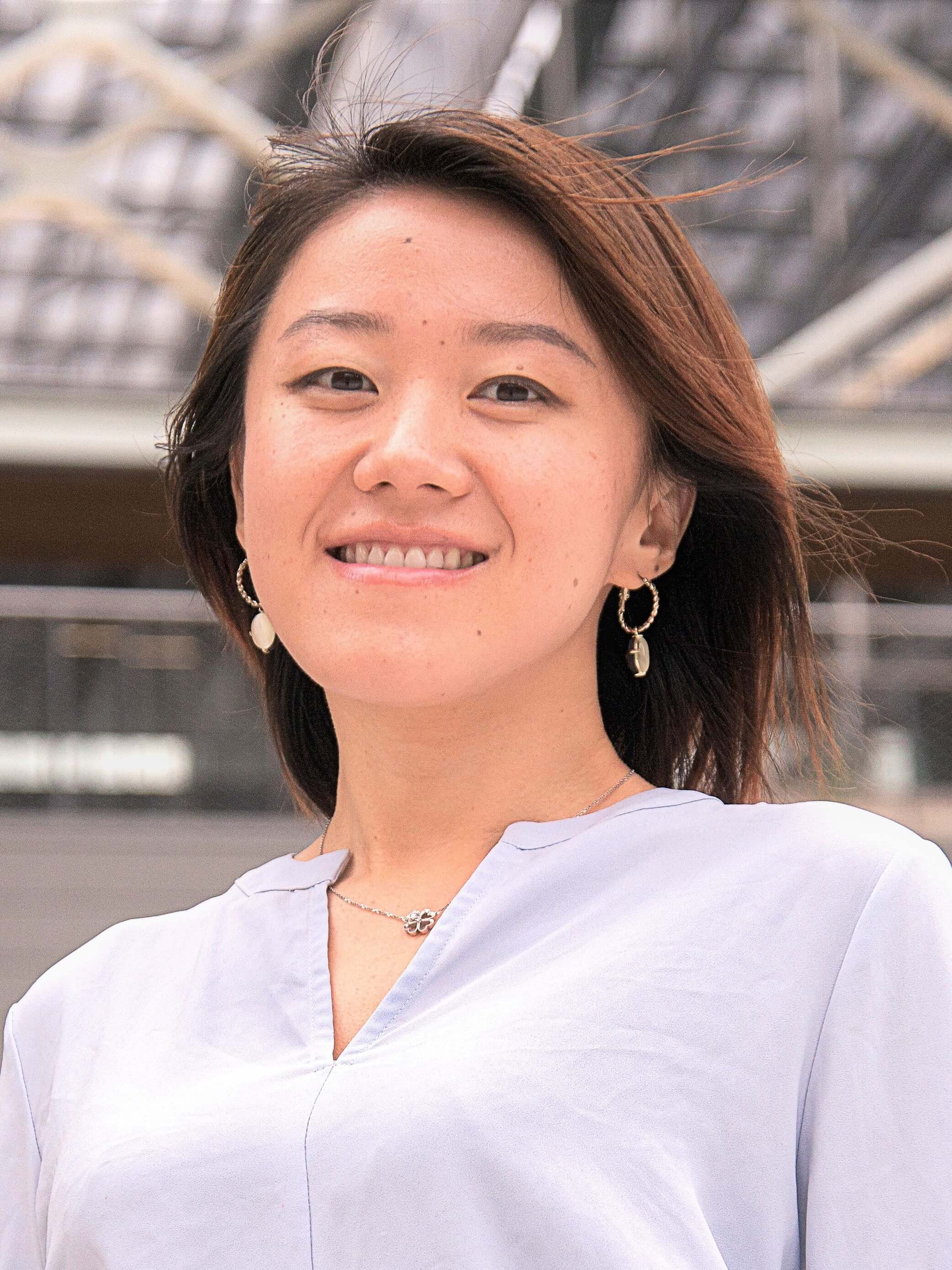}}]{Hao Zhu} (M'12--SM'19) is an Associate Professor of Electrical and Computer Engineering (ECE) at The University of Texas at Austin. She received the B.S. degree from Tsinghua University in 2006, and the M.Sc. and Ph.D. degrees from the University of Minnesota in 2009 and 2012. From 2012 to 2017, she was a Postdoctoral Research Associate and then an Assistant Professor of ECE at the University of Illinois at Urbana-Champaign. Her research focus is on developing algorithmic solutions for problems related to learning and optimization for future energy systems. Her current interest includes physics-aware and risk-aware machine learning for power system operations, and energy management system design under the cyber-physical coupling. She is a recipient of the NSF CAREER Award and an invited attendee to the US NAE Frontier of Engr.~(USFOE) Symposium, and also the faculty advisor for three Best Student Papers awarded at the North American Power Symposium. She is currently an Editor of \textit{IEEE Trans.~on Smart Grid} and \textit{IEEE Trans.~on Signal Processing}. 
\end{IEEEbiography}

\begin{IEEEbiography}[{\includegraphics[width=1in,
height=1.25in,keepaspectratio]{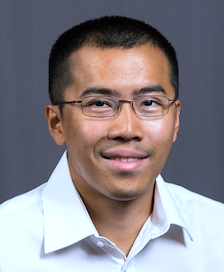}}]{Grani A. Hanasusanto} is  an Assistant Professor of Operations Research and Industrial Engineering at The University of Texas at Austin. He received an M.Sc. degree in Financial Engineering from the National University of Singapore and a Ph.D. degree in Operations Research from Imperial College London. Before joining UT Austin, he was a postdoctoral researcher at the College of Management of Technology at Ecole Polytechnique Federale de Lausanne. His research focuses on the design and analysis of tractable solution schemes for decision-making problems under uncertainty, with applications in operations management, energy systems, machine learning, and data analytics.
\end{IEEEbiography}

\end{document}